\documentclass[11pt]{article}
\usepackage[utf8]{inputenc}
\usepackage{amssymb,graphicx,mathtools,amsfonts}
\usepackage{amsmath}
\usepackage{amsthm}
\usepackage{hyperref}
\usepackage[dvipsnames]{xcolor}

\usepackage{a4wide}
\catcode`\@=11
\def\theequation{\@arabic{\c@section}.\@arabic{\c@equation}}
\catcode`\@=12
\graphicspath{{images/}}
\newtheorem{theorem}{Theorem}[section] 
\newtheorem{lemma}{Lemma}[section] 
\newtheorem{proposition}{Proposition}[section]

\newtheorem{remark}{Remark}[section]

\newcommand{\mc} {\mathcal}

\newcommand{\e}{\epsilon}

\newcommand{\al} {\alpha}
\newcommand{\ba} {\beta}
\newcommand{\de} {\delta}

\newcommand{\om} {\Omega}

\newcommand{\la} {\lambda}

\newcommand{\noi} {\noindent}
\newcommand{\na} {\nabla}

\newcommand{\ds} {\displaystyle}
\newcommand{\ra} {\rightarrow}
\newcommand{\real}{\mathbb{R}}

\newcommand{\rnn}{\mathbb{R}^{N}}
\newcommand{\lv}{\lVert}
\newcommand{\rv}{\rVert}
\newcommand{\weak}{\rightharpoonup}
\newcommand{\nlap}{\Delta_N}
\newcommand{\sobo}{W_0^{1,N}}

\newcommand{\grad}{\nabla}
\newcommand{\ntrl}{\mathbb{N}}
\newcommand{\ka}{\kappa}
\begin{document}

\title{Quasilinear Schr\"{o}dinger Equation involving Critical Hardy Potential and Choquard type Exponential nonlinearity}

\author{Shammi Malhotra\footnote{Department of Mathematics, Indian Institute of Technology Delhi, Hauz Khas New Delhi 110016,  India, maz228085@maths.iitd.ac.in}, Sarika Goyal\footnote{Department of Mathematics, Netaji Subhas University of Technology,
		Dwarka Sector-3, Dwarka, Delhi, 110078, India, sarika1.iitd@gmail.com, sarika@nsut.ac.in},  and K. Sreenadh\footnote{Department of Mathematics, Indian Institute of Technology Delhi, Hauz Khas New Delhi 110016,  India, sreenadh@maths.iitd.ac.in}}
\date{}

\maketitle
\begin{abstract}
\noi In this article, we study the following quasilinear Schr\"{o}dinger equation involving Hardy potential and Choquard type exponential nonlinearity with a parameter $\alpha$
\begin{equation*}
\left\{
\begin{array}{l}
 - \Delta_N w - \Delta_N(|w|^{2\al}) |w|^{2\alpha - 2} w - \lambda \frac{|w|^{2\alpha N-2}w}{\left( |x| \log\left(\frac{R}{|x|} \right) \right)^N} 
 = \left(\int_{\Omega} \frac{H(y,w(y))}{|x-y|^{\mu}}dy\right) h(x,w(x))\; \mbox{in }\; \Omega, \\
 w > 0 \mbox{ in } \Omega \setminus \{ 0\}, \quad \quad w = 0 \mbox{ on } \partial \Omega,
\end{array}
\right.
\end{equation*}
where $N\geq 2$, $\alpha>\frac12$, $0\leq \lambda< \left(\frac{N-1}{N}\right)^N$, $0 < \mu < N$, $h : \mathbb R^N \times \mathbb R \rightarrow \mathbb R$ is a continuous function with critical exponential growth in the sense of the Trudinger-Moser inequality and $H(x,t)= \int_{0}^{t} h(x,s) ds$ is the primitive of $h$. With the help of Mountain Pass Theorem and critical level which is obtained by the sequence of Moser functions, we establish the existence of a positive solution for a small range of $\lambda$. Moreover, we also investigate the existence of a positive solution for a non-homogeneous problem for every $0\leq \lambda <\left(\frac{N-1}{N}\right)^N.$ To the best of our knowledge, the results obtained here are new even in case of $N$-Laplace equation with Hardy potential. 

\medskip

\end{abstract}

\newpage
\section{Introduction}
In this article, we investigate the existence of a positive solution to the following quasilinear Schr\"{o}dinger equation with parameter having Hardy potential and critical Choquard exponential type nonlinearity 
\begin{equation*}\label{pq}
\tag{$ P_{\lambda}$}\left\{
\begin{array}{l}
 - \nlap w - \nlap(|w|^{2\al}) |w|^{2\alpha - 2} w - \lambda \frac{|w|^{2\alpha N-2}w}{\left( |x| \log\left(\frac{R}{|x|} \right) \right)^N} 
 = \left(\int_{\om} \frac{H(y,w(y))}{|x-y|^{\mu}}dy\right) h(x,w(x))\quad \mbox{in }\; \om, \\
 w > 0 \mbox{ in } \om \setminus \{ 0\}, \quad \quad w = 0 \mbox{ on } \partial \om,
\end{array}
\right.
\end{equation*}
where $\om$ is smooth bounded domain in $\rnn$ containing the origin, $\Delta_N(w) := div(|\grad w|^{N-2} \grad w )$ denotes the $N$-Laplacian, $N\geq 2$, $\al>\frac12$ is a parameter,  $ 0 < \lambda < \left( \frac{N-1}{N} \right)^N$,   $R \geq e^{2/N} \sup_{x \in \om} \{ |x| \}$, $0 < \mu < N$, and the function $H(x,s):=\int_{0}^{s} h(x,t) dt$, where $h(x,s)$ is a continuous function that behaves like $\exp\left(|s|^{\frac{2 \al N}{N-1}} \right)$ as $s \to \infty$, and $o(|s|^{\al N})$ as $s\to 0^+$. Moreover, $h$ satisfies some more suitable conditions which will be stated later.\\

\noi The operator present in our equation is known as the Quasilinear Schr\"{o}dinger operator and problem driven by these operators are known as Quasilinear Schr\"{o}dinger equations. Research on these equations has been active over the past two decades, with main contributions from Liu and Wang \cite{liu2003soliton} and Poppenberg et al. \cite{poppenberg2002existence}. These early studies focused on finding standing wave solutions for the equation:
\begin{equation}\label{shrodinger_equation}
\iota \partial_t w =  - \Delta w + V(x) w - f(|w|^2)w - \eta \Delta h(|w|^2) h^{\prime}(|w|^2) w.  \end{equation} 
\noi Here, $V(x)$ represents a potential with $x \in \rnn$; $f,h$ are real-valued functions, and $\eta \in \real$ is a constant. They proved the existence of a solution with the help of the constrained minimization argument as the associated functional is not well defined in the natural $H^1(\rnn)$ Sobolev space. \\
In 2004, the authors in \cite{colin2004solutions} introduced a novel variable substitution, enabling the study of this problem in the $H^1(\mathbb R^N)$ space and prove the existence of a solution. Perturbation methods are also employed to obtain the existence of a solution (see, for example, \cite{liu_critical_perturbation_method, liu_multiple_perturbation_method}). \\
The equation of the type \eqref{shrodinger_equation} represents a range of physical phenomena depending on the form of $h(s).$ For instance, with $h(s) = s$, Kurihara \cite{kurihara1981large} modeled the equation (\ref{shrodinger_equation}) as the superfluid film equation in plasma physics, whereas for $h(s) = (1+s)^{1/2}$, it can be used to model the self channelling of an ultra-high power laser in a matter \cite{borovskii1993dynamic, brandi1993relativistic}. \\

\noi Moreover, the operator in the problem \eqref{pq} is driven by a parameter $\alpha$. These parameter-operators are particularly relevant in plasma physics and fluid mechanics, where they play pivotal role in modeling soliton motions \cite{laedke1983evolution} and deforming of solitons in the presence of interactions \cite{nakamura1977damping}. Additionally, they appear in the study of Heisenberg ferromagnet's and magnons \cite{bass1990nonlinear, kosevich1990magnetic, lange1995time}, in dissipative quantum mechanics \cite{hasse1980general}, and as well as in condensed matter theory \cite{makhankov1984non}. \\
A variety of problems are based on this quasilinear operator with parameter $\alpha$. For instance, in \cite{do_soliton_R2, do_soliton_RN}, a specific choice of $\al = 1$ leads to the following equation:$$ - \Delta w + V(x) w - \Delta(w^2) w = |w|^{q-1} w + |w|^{p-1}w \mbox{ in } \rnn, $$ where $\la$ is a positive parameter, $3 < q < p \leq 2(2^*) - 1$ and $2^* = 2N/(N-2) $ is the critical Sobolev exponent (in dimension $N \geq 3$). The authors applied the concentration-compactness principle to establish the existence of solutions, beginning with dimension $N = 2$ and then extending to higher dimensions. \\

\noi The nonlinearity in problem \eqref{pq} is non-local and is commonly referred in the literature as Choquard-type nonlinearity. To handle such Choquard nonlinearity, we employ the Hardy-Littlewood-Sobolev (HLS) inequality, as stated in \cite{lieb2001analysis}:\\
For $\mu \in (0,N)$, $f \in L^p ( \rnn)$ and $g \in L^q ( \rnn)$ with $p,$ $q > 1$,  $\frac{1}{p} + \frac{\mu}{N} + \frac{1}{q} = 2$, there exists a constant $C_{p,q,N,\mu}>0$ such that
\begin{equation}\label{hls_inequality}
    \int_{\rnn} \int_{\rnn} \frac{f(x) \, g(y)}{|x - y|^{\mu}} \, dx\, dy  \leq  C_{p,q,N,\mu} \lVert f \rVert_{L^p (\rnn)} \lVert g \rVert_{L^q (\rnn)}.
\end{equation}
If $p = q = \frac{2 N}{2N - \mu}$ then $$C_{p,q,N,\mu} := C_{N,\mu} = \pi^{\frac{\mu}{2}} \frac{\Gamma\left( \frac{N}{2} - \frac{\mu}{2} \right)}{\Gamma\left(N - \frac{\mu}{2} \right)} \left[ \frac{\Gamma\left( \frac{N}{2} \right)}{\Gamma(N)} \right].$$ In this case, there is an equality in \eqref{hls_inequality} if and only if $f =$ (constant) $g$ and $$f(x) = c_0(a^2 + |x-b|^2)^{-\frac{(2N - \mu)}{2}}$$ for some $c_0 \in \mathbb{C}, 0 \neq a \in \real$ and $b \in \rnn$. \\

\noi One of the pioneering studies on Choquard nonlinearity was conducted by Pekar, who analyzed the equation $$ -\Delta u + u = (I_2 * |u|^2) u \mbox{ in } \mathbb{R}^3 $$ with $I_2$ being the Reisz potential. Pekar used this model to explore the interaction of free electrons with photons in an ionic lattice \cite{pekar1954untersuchungen}. There are other physical applications in which this nonlinearity appears, such as the collapse of galaxy fluctuations of scalar field dark matter \cite{guzman2003newtonian}, and Bose-Einstein condensation \cite{dalfovo1999theory}. Numerous researchers have studied Choquard nonlinearity problems in critical dimensions; see, for example, \cite{arora2019n, biswas2022quasilinear, de2014elliptic} and references therein. \\

\noi Moreover, the nonlinearity $h(x,s)$ is of critical growth in the sense of the following Trudinger-Moser inequality introduced in \cite{moser1971sharp} and stated as for $N\geq 2$, $w \in W^{1,N}_0(\om)$
\begin{equation}\label{TM-ineq}
\sup_{\|w\|\leq 1}\int_\om \exp(\beta|w|^{\frac{N}{N-1}})~dx < \infty
\end{equation}
if and only if $\beta \leq \al_N$, where $\al_N = N\omega_{N-1}^{\frac{1}{N-1}}$ and $\omega_{N-1}=$ $(N-1)$-dimensional surface area of $\mathbb{S}^{N-1}$. 

\noindent One of the initial works in critical dimension problem was done by Adimurthi in \cite{adimurthi1989positive}. The author showed the existence of solutions for Laplace operator in the unit disc of $\mathbb{R}^2$. Later in \cite{adimurthi1990existence}, Adimurthi extended the problem for the $N$-Laplacian operator by using the method of artificial constraint for both subcritical growth and critical growth of exponential nonlinearity. In $1994$, Figueiredo et. al \cite{de2014elliptic} studied the problem for the Laplacian operator in dimension $2$. However, they used the classical critical point theory along with the functions associated with the optimal Trudinger-Moser inequality. Using the same idea, the author \cite{marcos1996semilinear} generalized the problem for $N$-Laplacian. Additionally, multiplicity results for critical dimension problems have been established by various authors, including those in \cite{giacomoni2007global}, through the application of monotonicity and variational methods.\\

\noi Inspired by the above works, in this paper, we study quasilinear equations involving critical Choquard-type exponential nonlinearity with a parameter. Additionally, we incorporate a Hardy term, which adds the singular behaviour in our equation. It introduces further challenges to solve the problem. To address this, we utilize the following critical dimension Hardy inequality, as proved in \cite{sandeep2002existence}:
\begin{equation*}
\ds \int_{\om}  \frac{|w(x)|^N}{\left(|x| \log \left( \frac{R}{|x|}\right)\right)^N } dx\leq \frac{1}{C_N} \int_{\om} |\grad w(x) |^N dx \;\mbox{ for all }\; w \in \sobo(\om),
\end{equation*}
where $\om$ be any domain in $\rnn$ containing the origin, $R \geq e^{2/N} \ds\sup_{x \in \om} \{ |x| \}$ and $ C_N = \left( \frac{N-1}{N} \right)^N$ is the best constant. Many studies related to this inequality, along with additional applications, can be found in \cite{shen2004nonlinear, zhang2008class} and references therein.\\
The results presented are new in the sense that the problems considered has Hardy perturbation in the Schr\"{o}dinger equation combined with a non-local exponential nonlinearity and a parameter-driven operator. The presence of the parameter term, $\nlap(|w|^{2\al}) |w|^{2\alpha - 2} w$, imposes a critical exponential growth of the form $\exp\left(|s|^\frac{2 \al N}{N-1} \right)$, and the Hardy term takes the form $\frac{|w|^{2\al N -2}w}{\left(|x|\log \left( \frac{R}{|x|}\right)\right)^N} $. To establish the compactness threshold for Palais-Smale sequences, the base value $\tau$ at minimax is used and in the blow-up analysis, the growth behavior of the transformation was rigorously examined (see Proposition \ref{limit_polynomial_growth}). To the authors' knowledge, no such work is being done of this type. Utilizing a sequence of Moser functions to obtain threshold and the Mountain Pass Theorem yield the existence of the a solution. \\

\noi We assume that $h:\om \times \real \to \real $ to be of the form $h(x,s) = f(x,s) \exp{\left(|s|^{\frac{2\alpha N}{N-1}}\right)}$ and satisfy the following assumptions
\begin{enumerate}
    \item[$(h_1)$] $h:\om \times \real \to \real $ is continuous for all $x \in \overline{\om}$, and $h(x,s) = 0$ for $s \leq 0$ and $h(x,s) > 0 $ for all $s>0$.
    \item[$(h_2)$] There exists $R_0$,  $M_0 >0$ such that for all $s \geq R_0$ and $x \in \om$, we have $$ 0 < H(x,s) \leq M_0 h(x,s), $$ where $H(x,s) = \int_0^{s} h(x,t) \, dt$.
    \item[$(h_3)$] There exists $\theta_0 > \alpha N $ such that $$ 0 < \theta_0 H(x,s) \leq h(x,s)s \mbox{ for all }  (x,s) \in \om \times (0,\infty).$$ 
    \item[$(h_4)$] $\ds\lim_{s \to 0^+} \frac{H(x,s)}{|s|^{\al N}} = 0$ uniformly in $x \in \om.$
    \item[$(h_5)$] For any $\e > 0$ $$\displaystyle\lim_{t\to\infty}\sup_{x\in\overline\om} f(x,t)\exp\left(-\e|t|^{\frac{2\al N}{N-1}}\right)=0 \mbox{ and } \displaystyle\lim_{t\to\infty}\inf_{x\in\overline\om}f(x,t) {\exp(\e |t|^{\frac{2 \al N}{N-1}})}=\infty.$$
    \item[$(h_6)$] $\ds\lim_{s\to \infty} \frac{s h(x,s) H(y,s)}{\exp\left({2|s|^{\frac{2\alpha N}{N-1}}}\right)} = \infty$ uniformly in $x$, $y \in \om.$
\end{enumerate}
\begin{remark} One function which satisfies the above $(h_1)-(h_6)$ hypothesis is given by $h(x,s) = f(x,s) \exp(|s|^{\frac{2\al N}{N-1}})$, where $f(x,s)=\left\{\begin{array}{lr}
		{s^{a + {\al N}} \exp(d s^{r})},\; \mbox{if} \; s> 0\\
		0, \; \mbox{if} \; s\leq 0,	\end{array}\right.$ for $0<d\leq \al_N, 1 \leq r < \frac{2\al N}{N-1}$ and for some $ a > 0$.\end{remark}

\noindent Before mentioning the main result of this paper, we would like to mention that the range of the $\lambda$ for which the existence of the solution is shown to be dependent on the minimax level of the transform energy  functional $J_{0}: \sobo(\om) \to \real$, defined as:
\begin{equation*}
J_0(w)=\frac{1}{N}\displaystyle\int_{\om }|\nabla w|^{N}dx -\frac 12\int_{\om}\int_{\om} \frac{H(y,g(w(y)))H(x,g(w(x)))}{|x-y|^\mu}dxdy.
\end{equation*}
This functional satisfies the Mountain pass geometry and  the minimax level
\begin{equation}\label{minimax_level_at_0}
\tau := \inf_{\gamma \in \Gamma} \max_{t \in [0,1]} J_{0}(\gamma(t)).
\end{equation}
where $$ \Gamma := \{ \gamma \in C([0,1],\sobo(\om): \gamma(0) = 0, J_0(\gamma(1)) < 0 \}$$
is well-defined.
This $\tau$ is positive and the existence of critical point at this minimax level is proved in \cite{marcos1996semilinear}. With this introduction, we state our main result on the existence of a solution for (\ref{pq}):
\begin{theorem}\label{main_theorem}
     Suppose $N\geq 2$, $0 < \mu <N, R \geq e^{2/N} \sup_{x \in \om}\{|x| \}$, $\al > \frac{1}{2}$ and $(h_1)-(h_6)$ hold. Then the problem (\ref{pq}) possesses a non-trivial positive weak solution for all $\la< C_N\left(\frac{L - \tau}{L} \right)$, where $L = \frac{1}{2\al N}  \left( \frac{2N - \mu}{2N} \al_N \right)^{N-1}$, and $\tau$ is the minimax level defined in \eqref{minimax_level_at_0}.  
\end{theorem}
\begin{remark}
 It is not clear whether we can show the existence of a  positive solution of the problem (\ref{pq}) for $\la \in [\la_0, C_N)$ with $\lambda_0 = C_N\left(\frac{L - \tau}{L} \right)$. We leave this as an open question.
\end{remark}

\noindent However, for every $\la\in \left(0, (\frac{N-1}{N})^N\right)$, we obtain the existence of solution of the following non-homogeneous problem corresponding to (\ref{pq})  
\begin{equation}\label{pq_non_homo} 
\tag{$ P_{\lambda,NH}$}\left\{
\begin{array}{l}
 \mathcal{L}(w) - \lambda \frac{|w|^{2 \al N-2}w}{\left( |x| \log\left(\frac{R}{|x|} \right) \right)^N}
 = \left(\int_{\om} \frac{H(y,w(y))}{|x-y|^{\mu}}dy\right) h(x,w(x)) + \kappa q(x) \mbox{ in } \om, \\
w = 0 \mbox{ on } \partial \om,
\end{array}
\right.
\end{equation}
where $\mathcal{L}(w):=- \nlap w - \nlap(|w|^{2\al}) |w|^{2\alpha - 2} w$, $\om$ is smooth bounded domain in $\rnn$ containing the origin,  $N\geq 2$, $0 \leq \lambda < \left( \frac{N-1}{N}\right)^N$, $0 < \mu < N, \kappa > 0$, $ q \in (\sobo(\om))^*$, 
 $ R \geq e^{2/N} \sup_{x \in \om} \{ |x| \}, ~\al > \frac{1}{2}$ is a parameter. \\
\noindent We state the following existence theorem corresponding to the problem (\ref{pq_non_homo}): 
\begin{theorem}\label{main_theorem_non_homo}
     Suppose  $(h_1)-(h_6)$ holds. Then the problem (\ref{pq_non_homo}) possesses a non-trivial weak solution for small non-negative values of $\ka$.
\end{theorem}

\noindent The plan of the paper is as follows:
In section $2$, we first introduce the paper's notations and then define some basic spaces and the transformation require to study the problems (\ref{pq}) and (\ref{pq_non_homo}) variationally. Some properties of the transformation are also mentioned there. In section $3$, we prove the mountain pass geometry of the functional $J_{\lambda}$ associated with the problem (\ref{pq}). Also, we obtain the upper bound of the minimax level with the help of the Moser functions. Next, in section $4$, we see the behaviour of the Palais-Smale sequences related to the functional $J_{\lambda}$. We conclude the section by proving the Theorem \ref{main_theorem}. Finally, section $5$ discuss the corresponding non-homogeneous problem, and the proof of Theorem \ref{main_theorem_non_homo}. 
\section{Notations and Preliminaries}
Throughout the paper, we would like to use the following notations:
\begin{itemize}
    \item $A \subset \subset B$ means that $\bar{A}$ is compact in $B$. 
    \item The signs $\to, \weak$ denote the strong and weak convergence, respectively, in the corresponding spaces.
    \item If $A$ is any measurable set in $\rnn$, then $|A|$ denotes the $N$-dimensional Lebesgue measure.
    \item $B(x,R)$ denotes the $N$-dimensional ball in $\rnn$ centered at $x$ having radius $R$.
    \item $c, C_1, C_2,\cdots,$ and $ C', C'',\cdots$ , denote positive constants which may change from line to line. 
\end{itemize}

\noindent Let $1 \leq p < \infty$ and $u : \om \to \real$ be a measurable function. Then the Lebesgue space is denoted by $L^p(\om)$ and defined as: $$ L^p(\om) = \left\{ u : \om \to \real \mbox{ measurable } | \int_{\om} |u|^p~dx < \infty\right \} .$$ This is a Banach space endowed with the norm $\lv u \rv_{L^p(\om)}^p = \int_{\om} |u|^p ~ dx$. With the help of Lebesgue space, Sobolev space is denoted by $W^{1,N}(\om)$ and defined as $$ W^{1,N}(\om) = \left\{ u \in L^p(\om): \int_{\om} |\grad u|^p ~ dx < \infty  \right\}.$$
This is again a Banach space endowed with the norm $\lv u \rv_{W^{1,N}(\om)} := \left( \int_{\om} |u|^p + |\grad u|^p ~dx\right)^{\frac{1}{p}} $. One important subspace of this space is denoted by $\sobo(\om)$ and defined as the closure of $C_c^{\infty}(\om)$ functions with respect to the norm $\lv u \rv := \left( \int_{\om} |\grad u|^p ~dx\right)^{\frac{1}{p}}$.

\noindent The energy functional associated with the problem (\ref{pq}) is given by:
\begin{equation*}
\begin{split}
\tilde{J}_{\lambda}(w)=\frac{1}{N}\displaystyle\int_{\om }(1+(2\alpha)^{N-1}|w|^{N(2\al -1)})|\nabla w|^{N}dx & -\frac{\lambda}{2 \al N} \int_{\om} \frac{|w|^{2 \alpha N}}{\left( |x| \log\left( \frac{R}{|x|} \right) \right)^N}~dx \\ & -\frac 12\int_{\om}\int_{\om} \frac{H(y,w(y))H(x,w(x))}{|x-y|^\mu}dxdy.
\end{split}
\end{equation*}
Then one can see that the functional $\tilde{J}_{\lambda}$ is not well defined for all $w \in W_0^{1,N}(\om)$. So, to deal with it, we employ a transformation $w = g(u)$, which satisfies the following equation: 
\begin{equation}\label{g}
	\left\{
	\begin{array}{l}
	g^{\prime}(t)=\displaystyle\frac{1}{\left(1+(2\al)^{N-1}|g(t)|^{N(2\al-1)}\right)^{\frac{1}{N}}}~~\mbox{in}~~ [0,\infty),\\
	g(t)=-g(-t)~~\mbox{in}~~ (-\infty,0].
	\end{array}
	\right.
\end{equation}

\noindent The proof of the following lemma can be found in the following references \cite{colin2004solutions, li2015positive}.
\begin{lemma}\label{L1}
    The function $g$ satisfies the following properties:
\begin{itemize}
			\item[$(g_1)$] $g$ is uniquely defined, $C^{\infty}(\real)$ and invertible;
			\item[$(g_2)$] $g(0)=0$;
			\item[$(g_3)$] $0<g^{\prime}(t)\leq 1$ for all $t\in \mathbb{R}$;
			\item[$(g_4)$] $\frac{1}{2}g(t)\leq \al t g^{\prime}(t)\leq \al g(t)$ for all $t>0$;
			\item[$(g_5)$] $|g(t)|\leq |t|$ for all $t\in \mathbb{R}$;
			\item[$(g_6)$] $|g(t)|\leq (2\al)^{\frac{1}{2N\al}}|t|^{\frac{1}{2\al}}$ for all $t\in \mathbb{R}$;
			\item[$(g_7)$] $\displaystyle \lim_{t \to \infty} \frac{g(t)}{t^{\frac{1}{2\al}}} = (2 \alpha)^{\frac{1}{2N\alpha}}$;
			\item[$(g_8)$]  $|g(t)|\geq g(1)|t|$ for $|t|\leq 1$ and $|g(t)|\geq g(1)|t|^{\frac{1}{2\al}}$ for $|t|\geq 1$;
			\item[$(g_9)$] $g^{\prime \prime}(t)<0$ when $s>0$ and $g^{\prime \prime}(t)>0$ when $s<0$;
            \item[$(g_{10})$] $|g(t)|^{2 \al -1} g^{\prime}(t)\leq \frac{1}{(2 \al)^{\frac{N-1}{N}}}$ for all $t \in \real$.
		\end{itemize}
\end{lemma}
\noindent The behaviour of $g$ in $(g_7)$ can be calculated and will be crucial in proving one of the important results. 
\begin{proposition} \label{limit_polynomial_growth}
Let $g$ be as in \eqref{g}. Then for any $k \in \ntrl \setminus \{1\}$ and $k \geq \left( 2 \left[ 1 - \left( \frac{B}{1+B} \right)^{\frac{1}{N}} \right] \right)^{-1}$ with $B =(2\al)^{N-1} (g(1))^{(2\al - 1)N}$ , we have $(g(s))^{2 \al } \geq (1-\frac{1}{k}) (2\al)^{\frac{1}{N}} s$ for all $s \geq S_k$, where $S_k$ has polynomial growth in $k$.
\end{proposition}
\begin{proof}
For $s \geq 1$, using $(g_8)$ we obtain
\begin{align*}
g^{\prime}(s) & =\frac{1}{\left(1+(2\al)^{N-1}(g(s))^{N(2\al-1)}\right)^{\frac{1}{N}}} \\
& = \left(\frac{(2\al)^{N-1}(g(s))^{N(2\al-1)}}{1+(2\al)^{N-1}(g(s))^{N(2\al-1)}} \right)^{\frac{1}{N}} \frac{1}{(2\al)^{\frac{N-1}{N}}(g(s))^{(2\al-1)}}\\
& \geq \left( \frac{(2\al)^{N-1} (g(1))^{(2\al -1 )N} s^{\frac{(2\al - 1)N}{2\al}}}{1 + (2\al)^{N-1} (g(1))^{(2\al -1 )N} s^{\frac{(2\al - 1)N}{2\al}}} \right)^{\frac{1}{N}} \frac{1}{(2\al)^{\frac{N-1}{N}}(g(s))^{(2\al-1)}} \\
& \geq \left(1 - \frac{\e}{2} \right)  \frac{1}{(2\al)^{\frac{N-1}{N}}(g(s))^{(2\al-1)}}, 
\end{align*}
where the last inequality holds for all $ s \geq \max \left\{ \left( \frac{(1 - \frac{\e}{2})^N}{(1 - (1 - \frac{\e}{2})^N )(2\al)^{N-1} (g(1))^{(2\al -1 )N}} \right)^{\frac{2 \al}{(2\al - 1) N}} ( = s_{\e/2} ), 1\right\}$ (comes from the direct computation). Also, note that $ s_{\e/2} \geq 1$ for $\e \leq 2 \left[ 1 - \left( \frac{B}{1+B} \right)^{\frac{1}{N}} \right] $. Thus, we obtain
\begin{align*}
    \frac{d}{ds} \left( \frac{g(s)^{2\al}}{(2\al)^{\frac{1}{N}}} \right) \geq \left(1 - \frac{\e}{2} \right). 
\end{align*}
Now, integrating $s_{\e/2}$ to s gives 
\begin{align*}
    \frac{g(s)^{2\al}}{(2 \al)^{\frac{1}{N}}} \geq ~ & \frac{g(s_{\e/2})^{2\al}}{(2 \al)^{\frac{1}{N}}} + \left(1 - \frac{\e}{2}\right) ( s- s_{\e/2})\\
\geq ~ & \left(1 - \frac{\e}{2}\right) ( s- s_{\e/2}) .   
\end{align*}
So, the right hand side is greater than $(1-\e)$ if and only if $ s \geq \left( \frac{2}{\e} - 1 \right) s_{\e/2} ~ ( = S_\e )$. Thus, if we take $\e = \frac{1}{k}$, then $S_\e$ has polynomial growth in $k$ as $s_\e $ has polynomial growth in $k$. This proves the claim. 
\end{proof}
\noindent After the transformation, the new energy functional $J_{\la}: W^{1,N}_{0}(\om)\rightarrow \real$ associated with $(P_{\la})$ is defined as:
\small{\begin{align}\label{new_energy}
J_{\lambda}(u)=& \frac{1}{N}\displaystyle\int_{\om }|\nabla u|^{N}dx -\frac{\lambda}{2 \alpha N} \int_{\om} \frac{|g(u)|^{2\alpha N}}{\left( |x| \log\left( \frac{R}{|x|} \right) \right)
^N } dx  -\frac 12\int_{\om}\int_{\om} \frac{H(y,g(u(y)))H(x,g(u(x)))}{|x-y|^\mu}dxdy. 
\end{align}}

\noindent It is evident that the functional is well-defined, $C^1(W_0^{1,N}(\om),\real)$, and its derivative is expressed as
\begin{align*}
\langle J_{\lambda}^{\prime}(u), v\rangle=& \displaystyle\int_{\om }|\nabla u|^{N-2} \na u  \na v \, dx - \lambda \int_{\om} \frac{|g(u)|^{2\alpha N-2} g(u) g^{\prime}(u) v}{\left( |x| \log\left( \frac{R}{|x|} \right) \right)^N}~dx  \notag \\ &- \int_{\om}\int_{\om} \frac{H(y,g(u(y)))h(x,g(u(x))) g^{\prime}(u(x))v(x)}{|x-y|^\mu}dxdy.
\end{align*}
As $u \in \sobo(\om)$, in view of Sobolev embedding and $(g_5)$, we have $ g(u) \in L^q(\om)$ for any $q \geq 1$. This together with $(h_4)$ and $(h_5)$ implies that $H(x,g(u(x))) \in L^q(\om)$ for any $q\geq 1$. Indeed,  
\begin{align}\label{exponential_term_bound}
|H(x,t)| \le \e |t|^{ \al N } + C(N,\e, r , \alpha ) |t|^r \exp\left((1+\e)|t|^{\frac{2\al N}{N-1}}\right),\;\; \text{for all}\; (x,t)\in \om \times \real.
\end{align}
for any $r \geq 1$ and some constant $C(N,\e,r, \al) > 0$. From this, one deduces that $|H(x,g(u(x)))| \in L^q(\om)$. Moreover, with help of (\ref{hls_inequality}), we obtain 
\begin{equation*}
    \int_{\om}\int_{\om} \frac{H(y,g(u(y)))H(x,g(u(x)))}{|x-y|^\mu}dxdy \leq C \lv H(., g(u(.))) \rv_{L^{\frac{2N}{2N - \mu}}}^2.
\end{equation*}
The critical points of the functional $J_{\lambda}$ are the weak solutions of the following problem:
\begin{equation*}\label{pq_new}
\tag{$ P_{\lambda}^*$}\left\{
\begin{array}{l}
 - \nlap u  - \lambda \frac{|g(u)|^{2\alpha N-2}g(u) g^{\prime}(u)}{\left( |x| \log\left(\frac{R}{|x|} \right) \right)^N} 
 = \left(\int_{\om} \frac{H(y,g(u(y)))}{|x-y|^{\mu}}dy\right) h(x,g(u(x))g^{\prime}(u(x)) \mbox{ in } \om, \\
 u > 0 \mbox{ in } \om \setminus \{ 0\}, \quad \quad u = 0 \mbox{ on } \partial \om.
\end{array}
\right.
\end{equation*}
As one can easily see that the problems (\ref{pq}) and (\ref{pq_new}) are equivalent. Moreover, $u$ is a weak solution of (\ref{pq_new}) if and only if $w = g(u)$ is a weak solution of (\ref{pq}).

\section{ Mountain Pass Geometry and Critical Level}
In this section, we will show that the functional $J_{\lambda}$ satisfies the mountain pass geometry and we also calculate the upper bound of the minimax level with the help of Moser functions. To prove the existence of a solution to the problem (\ref{pq}), we will employ the variant of the Mountain Pass Theorem \cite{rabinowitz1986minimax}. First, we show that the functional possesses the mountain pass geometry. 
\begin{lemma}\label{geometry_of_functional_at_infinity}
    For any $u \in \sobo (\om)$, we have $J_{\lambda}(tu) \to - \infty$ as $t \to \infty$.
\end{lemma}
\begin{proof}
From the assumption $(h_2)$ of $H$, we deduce that there exists $C>0$ such that for all $s \geq \max\{R_0,g(1),1\}$ and $x \in \om$, we have $$H(x,s) \geq C \exp{\left(\frac{s}{M_0}\right)}.$$
Using this together with the property $(g_8)$ of Lemma \ref{L1}, we have $$H(x,g(u)) \geq \frac{C |g(u)|^k}{M_0^k k!} - C_2 \geq \frac{C (g(1))^k}{M_0^k k!} |u|^{\frac{k}{2\alpha}} - C_2 = C_1 |u|^{\frac{k}{2\alpha}} - C_2, $$ for any $u \in \sobo(\om), k \in \ntrl$ and constants $ C_1$ and $ C_2 > 0$ depends on $k$. So, we obtain
\begin{align*}
    \int_{\om} \left(\int_{\om} \right. & \left. \frac{H(y, g(tu
			))}{|x-y|^{\mu}}dy\right)H(x, g(t u))\,dx \\
   &  \geq \int_{\om} \int_{\om} \frac{(C_1 (t |u(y)|)^{\frac{k}{2\alpha}}-C_2)(C_1 (t |u(x)|))^{\frac{k}{2\alpha}}-C_2)}{|x-y|^{\mu}}\,dxdy\\
    & = C_1^2  \int_{\om} \int_{\om} \frac{(t|u(y)|)^{\frac{k}{2\alpha}} (t|u(x)|)^{\frac{k}{2\alpha}}}{|x-y|^{\mu}}~dxdy  -2C_1C_2\int_{\om} \int_{\om}\frac{(t|u(x)|)^{\frac{k}{2\alpha}}}{|x-y|^{\mu}}~dxdy + C_2^2 \int_{\om} \int_{\om} \frac{dx ~ dy}{|x-y|^{\mu}}\\
    & = C_1^2 t^{\frac{k}{\alpha}}  \int_{\om} \int_{\om} \frac{|u(y)|^{\frac{k}{2\alpha}} |u(x)|^{\frac{k}{2\alpha}}}{|x-y|^{\mu}}~dxdy -2C_1C_2 t^{\frac{k}{2\alpha}}\int_{\om} \int_{\om}\frac{|u(x)|^{\frac{k}{2\alpha}}}{|x-y|^{\mu}}~dxdy + C_2^2 \int_{\om} \int_{\om} \frac{dx ~ dy}{|x-y|^{\mu}}.
\end{align*}
Using this estimate, we deduce
\begin{align*}
    J_{\lambda}(tu) \leq & \frac{t^N}{N} \int_{\om} | \nabla u|^N - \frac{\lambda t^{2 \alpha N}}{2 \alpha N} \int_{\om}\frac{|g(u)|^{2 \alpha N}}{\left( |x| \log\left( \frac{R}{|x|} \right) \right)^N} dx - \frac{1}{2} \int_{\om} \left(\int_{\om} \frac{H(y, g(tu))}{|x-y|^{\mu}}dy\right)H(x, g(t u))~dx \\
    \leq & \frac{t^N}{N} \int_{\om} | \nabla u|^N - \frac{C_1^2t^{\frac{k}{\alpha}}}{2}   \int_{\om} \int_{\om} \frac{|u(y)|^{\frac{k}{2\alpha}} |u(x)|^{\frac{k}{2\alpha}}}{|x-y|^{\mu}}~dxdy + C_1C_2 t^{\frac{k}{2\alpha}}\int_{\om} \int_{\om}\frac{|u(x)|^{\frac{k}{2\alpha}}}{|x-y|^{\mu}}~dxdy \\
    & - \frac{C_2^2}{2} \int_{\om} \int_{\om} \frac{dx ~ dy}{|x-y|^{\mu}},
\end{align*}
choosing $k > N \alpha$, yields $J_{\lambda}(t u) \to - \infty$ as $t \to \infty$.
\end{proof}

\noindent Next, we show that the energy functional $J_{\lambda}$ possesses the mountain pass geometry near the origin.
\begin{lemma}\label{geometry_of_functional_near_zero}
Suppose $(h_4)$ and $(h_5)$ hold. Then there exists $\delta, \rho > 0$ such that $$ J_{\lambda}(u) \geq \delta \mbox{ for all } \lv u \rv  = \rho.$$
\end{lemma}
\begin{proof}
From \eqref{exponential_term_bound}, $(g_5)$, Minkowski inequality, H\"{o}lder inequality and Sobolev inequality, we have
\begin{align*}
    &\left( \int_{\om} |H(x,g(u))|^{\frac{2N}{2N - \mu}} \right)^{\frac{2N - \mu}{N}} \notag \\
    & \leq \left[ \int_{\om} \left\{ \e |u|^{\al N} + C(N,\e, r, \al) |u|^r \exp\left((1+\e)|g(u)|^{\frac{2\al N}{N-1}}\right) \right\}^{\frac{2N}{2N-\mu}} \right]^{\frac{2N - \mu}{N}} \notag \\
    & \leq \left[ \left( \int_{\om} \left( \e |u|^{\al N} \right)^{\frac{2N}{2N-\mu}} \right)^{\frac{2N - \mu}{2N}}+ C_2 \left\{ \int_{\om} \left( |u|^r \exp\left((1+\e) (2\alpha)^{\frac{1}{N-1}} |u|^{\frac{N}{N-1}} \right) \right)^{\frac{2N}{2N - \mu}} \right\}^{\frac{2N - \mu}{2N}} \right]^2 \notag \\
    & \leq  2 \left[ \int_{\om} \left( \e |u|^{\al N} \right)^{\frac{2N}{2N-\mu}} \right]^{\frac{2N - \mu}{N}}+ C_2 \left[ \int_{\om} \left( |u|^r \exp\left((1+\e) (2\alpha)^{\frac{1}{N-1}} |u|^{\frac{N}{N-1}} \right) \right)^{\frac{2N }{2N- \mu}} \right]^{\frac{2N - \mu}{N}}  \notag \\
    & \leq 2 \e^2 \left[ \int_{\om} |u|^{\frac{2 \al N^2}{2N-\mu}} \right]^{\frac{2N-\mu}{N}} + C_2\left( \int_{\om} |u|^{\frac{4rN}{2N -\mu}} \right)^{\frac{2N-\mu}{2N}} \left[ \int_{\om} \left\{ \exp\left((1+\e)  (2\alpha)^{\frac{1}{N-1}} |u|^{\frac{N}{N-1}} \right) \right\}^{\frac{4N}{2N -\mu}} \right]^{\frac{2N - \mu}{2N}} \\
    & \leq C_1 \lv u \rv^{2 \al N} +  C_2  \lv u\rv^{2r}_{L^{\frac{4rN}{2N -\mu}}}  \left[ \int_{\om} \exp\left\{(1+\e)  (2\alpha)^{\frac{1}{N-1}} \lv u\rv^{\frac{N}{N-1}} \frac{4N}{2N -\mu}  \left( \frac{|u|^{\frac{N}{N-1}}}{ \lv u\rv^{\frac{N}{N-1}}} \right) \right\} \right]^{\frac{2N - \mu}{2N}}.
\end{align*}
where $C_1$ and $C_2$ are constants depending on $N, \e, r,$ and $ \al$.

Now, if we choose $\lv u \rv$ is small enough such that $$ (1+\e)  (2\alpha)^{\frac{1}{N-1}} \lv u\rv^{\frac{N}{N-1}} \frac{4N}{2N -\mu} < \al_N .$$
Then the Trudinger-Moser inequality implies 
\begin{equation}\label{origin_eqn}
    \left( \int_{\om} (H(x,g(u)))^{\frac{2N}{2N - \mu}} \right)^{\frac{2N - \mu}{N}} \leq C_1 \lv u \rv ^{2 \al N} + C_2 \lv u \rv ^{2r}. 
\end{equation}
Applying $(g_6), (g_{10})$, Hardy inequality, and (\ref{origin_eqn}), we deduce $$ J_{\lambda}(u) \geq \frac{1}{N}\left( 1 - \frac{\lambda}{C_N}\right) \lv u \rv ^N -  C_1 \lv u \rv ^{2 \al N} - C_2 \lv u \rv ^{2r}.$$
Choosing $2r > N$, we conclude that there exists $\rho>0$ such that $J_{\lambda}(u) \geq \delta$ for all $\lv u \rv  = \rho$.
\end{proof}


\noindent Now, using Moser functions, we show that the minimax value is less than a specific value. Moreover, as we will see later, that specific value turns out to be the first threshold level below which the Palais-Smale sequences possess the convergent subsequences. First of all, let us define the minimax level $\beta^*$. Let $$ \Gamma_{\lambda} = \{ \gamma \in C([0,1],\sobo(\om): \gamma(0) = 0, J_{\lambda}(\gamma(1)) < 0 \}$$ be the set of all paths, then the minimax level is defined as $$\beta^*(\lambda) = \inf_{\gamma \in \Gamma} \max_{t \in [0,1]} J_{\lambda}(\gamma(t)).$$ Note that, $\beta^*(\lambda)$ is decreasing with respect to $\lambda$. Consequently, as $\la \to 0$, we have $\beta^*(\lambda) \to \beta^*(0)=\tau$, where $\tau$ is defined as in \eqref{minimax_level_at_0}.
\begin{lemma}
If $(h6)$ holds, then for all $\lambda \in [0, C_N)$, we have  $$ 0 < \beta^*(\lambda) < \frac{1}{2\al N}  \left( \frac{2N - \mu}{2N} \al_N \right)^{N-1},$$where $\al_N = N \omega_{N-1}^{1/{N-1}}$ and $\omega_{N-1}$ denotes the $N-1$ dimensional measure of surface measure of unit ball in $\rnn$.
\end{lemma}
\begin{proof}
For any non-trivial non-negative $u \in \sobo(\om)$, we know from Lemma \ref{geometry_of_functional_at_infinity} that $J_{\lambda}(tu) \to -\infty$ as $t \to \infty$. Thus $$ \beta^*(\lambda) \leq \max_{t \in [0,1]} J_{\lambda}(\gamma(t))$$ for all $u \in \sobo(\om) \setminus \{ 0 \}$ with $J_{\lambda}(u) < 0$. \\
So, it is enough to show that there exists $u_0 \in \sobo(\om)$ such that $\lv u_0 \rv = 1$ and 
\begin{equation}\label{less_than_level}
\max_{t \in [0,\infty)} J_{\lambda}(t u_0) < \frac{1}{2 \alpha N} \left( \frac{2N - \mu}{2N} \al_N \right)^{N-1}.    
\end{equation}
Consider the Moser functions defined as 
\begin{equation*}
w_k(x)=\frac{1}{\omega_{N-1}^{\frac{1}{N}}}\left\{
\begin{split}
& (\log k)^{\frac{N-1}{N}},\; 0\leq |x|\leq \frac{\rho}{k},\\
& \frac{\log \left(\frac{\rho}{|x|}\right)}{(\log k)^{\frac{1}{N}}}, \; \frac{\rho}{k}\leq |x|\leq \rho\\
& 0,\; |x|\geq \rho.
\end{split}
\right.
\end{equation*}
So, $Supp(w_k) \subset B(0,\rho)$ and $\lv w_k \rv = 1$ for all $k \in \mathbb{N}$. We claim that the (\ref{less_than_level}) holds for some $w_k$. If not, then for all $k \in \mathbb{N}$, there exists $t_k > 0$ such that 
\begin{equation}\label{moser_lower_bound}
    \max_{t \in [0,\infty)} J_{\lambda}(t w_k) = J_{\lambda}(t_k w_k) \geq \frac{1}{2\al N} \left( \frac{2N - \mu}{2N} \al_N \right)^{N-1}
\end{equation}
and 
\begin{equation}\label{derivative_zero_at_maximum_of_tk}
  \left.  \frac{d}{dt} ( J_{\lambda}(t w_k)) \right|_{t = t_k} = 0.
\end{equation}
As $J_{\lambda}(tw_k) \to - \infty$ as $t \to \infty$ uniformly in $k$. Then, in view of (\ref{moser_lower_bound}), $\{ t_k\}$ is a bounded sequence. \\
As $H(x,s)$ is non negative, we have $$ J_{\lambda}(t_k w_k) \leq \frac{1}{N}  \int_{\om} | \grad (t_k w_k)|^N = \frac{1}{N}  t_k^{N}. $$
Thus, we have $$ \frac{1}{N} t_k^{N} \geq \frac{1}{2\al N}  \left( \frac{2N - \mu}{2N} \al_N \right)^{N-1}, $$
which implies 
\begin{equation*}
    t_k^N \geq \frac{1}{2\al } \left( \frac{2N - \mu}{2N} \al_N \right)^{N-1}.
\end{equation*}
Thus, $\{t_k\}$ is bounded below by a positive number. \\
Also, from this, we deduce that in the ball $B(0,\rho/k)$, we have 
\begin{equation}\label{moser_functions_inside_ball_behaviour}
    t_k w_k(x) = \frac{t_k (\log k)^{(N-1)/N}}{w_{N-1}^{1/N}} \to \infty  \mbox{ as } k \to \infty.
\end{equation} 
We claim that $t_k ^N \to \frac{1}{2\al } \left( \frac{2N - \mu}{2N} \al_N \right)^{N-1}$ as $k \to \infty$. If not, then there exists a small $\delta > 0$ and $n_{\delta} \in \mathbb{N}$ such that
\begin{equation}\label{minimax_level_tk_not_limit}
     t_k^{\frac{N}{N-1}} \geq \frac{1}{(2\al)^{1/(N-1)} } \left( \frac{2N - \mu}{2N} \al_N \right)(1+\delta) \mbox{ for all } k \geq n_{\delta}.
\end{equation}
Now, \eqref{derivative_zero_at_maximum_of_tk} and $(g_4)$ implies that 

\begin{align*}
t_k^N = & \lambda \int_{\om} \frac{|g(t_k w_k)|^{2\al N-2} g(t_k w_k) g^{\prime}(t_k w_k) t_k w_k}{\left( |x| \log\left( \frac{R}{|x|} \right) \right)^N} \\ 
& + \int_{\om}\int_{\om} \frac{H(y,g(t_k w_k(y)))h(x,g(t_k w_k(x))) g^{\prime}(t_k w_k(y)) t_k w_k(x)}{|x-y|^\mu}dxdy \\    
\geq & \frac{\lambda}{2 \al } \int_{\om} \frac{|g(t_k w_k)|^{2\al N} }{\left( |x| \log\left( \frac{R}{|x|} \right) \right)^N} \\
& + \int_{\om}\int_{\om} \frac{H(y,g(t_k w_k(y)))h(x,g(t_k w_k(x))) g^{\prime}(t_k w_k(y)) t_k w_k(x)}{|x-y|^\mu}dxdy \\
\geq & \frac{1}{2 \al} \int_{B(0, \rho/k)}\int_{B(0, \rho/k)} \frac{H(y,g(t_k w_k(y)))h(x,g(t_k w_k(x))) g(t_k w_k(y)) }{|x-y|^\mu}dxdy \\
\end{align*}
which is independent of $\lambda$. \\
The assumption $(h_6)$ yield that for each $d > 0$ there exists $s_d$ such that 
\begin{equation*}
    s h(x,s) H(y,s) \geq d \exp\left( 2 |s|^{\frac{2 \al N }{N-1}}\right) \mbox{ whenever } s \geq s_d.
\end{equation*}
Also, (\ref{moser_functions_inside_ball_behaviour}) and $(g_8)$ implies that we can choose $r_d \in \mathbb{N}$ such that for all $ k \geq r_d$, $$ g \left( t_k w_k \right) \geq g(1) (t_k w_k)^{1/(2\al)} \geq s_d$$ uniformly in $B(0,\rho/k)$. Hence, for all $k \geq r_d$, we deduce
\begin{equation}\label{minimax_level_eq1}
    H(y,g(t_k w_k(y)))h(x,g(t_k w_k(x))) g(t_k w_k(y)) \geq d \exp \left( 2 |g(t_k w_k)|^{\frac{2 \al N }{N-1}} \right). 
\end{equation}
Now, from \cite{alves2016existence}, we obtain 
\begin{equation}\label{minimax_level_eq2}
    \int_{B(0, \rho/k)}\int_{B(0, \rho/k)} \frac{~dxdy}{|x-y|^{\mu}} \geq C_{\mu, N} \left(\frac{\rho}{k}\right)^{2N-\mu},
\end{equation}
where $C_{\mu, N}$ is a positive constant depending on $\mu$ and $N$. \\
Moreover, by $(g_7)$, for any $\e > 0$ there exists $k_0(\e) \in \ntrl$ such that for all $k \geq k_0(\e)$, we have 
\begin{equation}\label{minimax_level_eq3}
|g(t_k w_k) |^{\frac{2\al N }{N-1}} \geq ((2\al)^{1/(N-1)} - \e) (t_k w_k)^{\frac{N}{N-1}}.
\end{equation}
Thus, using \eqref{minimax_level_eq1}, \eqref{minimax_level_eq2}, \eqref{minimax_level_eq3} and \eqref{minimax_level_tk_not_limit} we have for $k \geq \max\{ n_\delta, r_d, k_0(\e)\}$
\begin{align*}
    t_k^N & \geq \frac{d}{2\al}   \exp\left( 2 \left\{ g \left( \frac{t_k (\log k)^{\frac{N-1}{N}}}{\omega_{N-1}^{1/N}} \right)\right\}^{\frac{2 \al N }{N-1}}\right) C_{\mu,N} \left(\frac{\rho}{k}\right)^{2N-\mu} \\
    & \geq \frac{d}{2\al}   \exp\left( 2 ( (2 \al)^{\frac{1}{N-1}} - \e) \left( \frac{t_k (\log k )^{\frac{N-1}{N}}}{\omega_{N-1}^{1/N}} \right)^{\frac{N}{N-1}}\right) C_{\mu,N} \left(\frac{\rho}{k}\right)^{2N-\mu} \\
    & \geq \frac{d}{2\al} C_{\mu,N} \left(\frac{\rho}{k}\right)^{2N-\mu} \exp\left( 2 ( (2 \al)^{\frac{1}{N-1}} - \e) \frac{\log k}{\omega_{N-1}^{\frac{1}{N-1}}} t_k^{\frac{N}{N-1}}\right)  \\
   & \geq \frac{d}{2\al} C_{\mu,N} \left(\frac{\rho}{k}\right)^{2N-\mu} \exp\left( 2 ( (2 \al)^{\frac{1}{N-1}} - \e) \frac{\log k}{\omega_{N-1}^{\frac{1}{N-1}}} \frac{1}{(2\al)^{1/(N-1)} } \left( \frac{2N - \mu}{2N} \al_N \right)(1+\delta) \right)  \\
   & \geq \frac{d}{2\al} C_{\mu,N} \left(\frac{\rho}{k}\right)^{2N-\mu} \exp\left(  \frac{(1+\delta)( (2 \al)^{\frac{1}{N-1}} - \e)}{(2\al)^{\frac{1}{N-1}}} (2N - \mu) \log k  \right).
\end{align*}
Choosing $\e > 0$ such that $\frac{(1+\delta)\left( (2 \al)^{\frac{1}{N-1}} - \e\right)}{(2\al)^{\frac{1}{N-1}}} > 1 + \gamma$ for some $\gamma > 0$ , we obtain
\begin{align*}
    t_k^N &\geq \frac{d}{2\al} C_{\mu,N} \left(\frac{\rho}{k}\right)^{2N-\mu} \exp( (1+\gamma)(2N - \mu) \log k) \\
t_k^N & \geq \frac{d}{2\al} C_{\mu,N} \rho^{2N - \mu} k^{\gamma ( 2N - \mu)},
\end{align*}
which is a contradiction to the fact that $\{t_k\}$ is bounded as right hand side tends to infinity as $k \to \infty$. Therefore, $t_k  \to \frac{1}{(2\al)^{\frac{1}{N}} } \left( \frac{2N - \mu}{2N} \al_N \right)^{\frac{N-1}{N}}$.\\
Now, using the lower bound of $t_k$, we obtain that 
\begin{align*}
    t_k^N & \geq d C_{\mu,N} \left(\frac{\rho}{k}\right)^{2N-\mu} \exp\left( 2 ( (2 \al)^{\frac{1}{N-1}} - \e) \frac{\log k}{\omega_{N-1}^{\frac{1}{N-1}}} t_k^{\frac{N}{N-1}}\right)  \\
& \geq d C_{\mu,N} \left(\frac{\rho}{k}\right)^{2N-\mu} \exp\left(  \frac{( (2 \al)^{\frac{1}{N-1}} - \e)}{(2\al)^{\frac{1}{N-1}}} (2N - \mu) \log k  \right)  \\   
& \geq d C_{\mu,N} \rho^{2N - \mu} \exp\left( \left( \frac{ (2 \al)^{\frac{1}{N-1}} - \e}{(2\al)^{\frac{1}{N-1}}} - 1\right) (2N - \mu) \log k  \right)  \\   
& \geq d C_{\mu,N} \rho^{2N - \mu} \exp\left( \frac{-\e}{(2\al)^{\frac{1}{N-1}}}  (2N - \mu) \log k  \right).  \\   
\end{align*}
So, from the Proposition \ref{limit_polynomial_growth}, if we take $\e = \frac{1}{m}$, then $k_m$ has polynomial growth in $m$. Therefore we have 
$$t_{k_m}^N \geq d C_{\mu,N} \rho^{2N - \mu} \exp\left( -\frac{(2N - \mu)\log k_m}{m(2\al)^{\frac{1}{N-1}}}  \right) $$
and letting $m \to \infty$, we get that $$ \frac{1}{2\al } \left( \frac{2N - \mu}{2N} \al_N \right)^{N-1} \geq d C_{\mu, N} \rho^{2N - \mu}.$$ But, we can choose $d$ as large as possible, we have a contradiction. Thus, there exists some $k$ such that (\ref{moser_lower_bound}) holds. This proves our assertion.
\end{proof}

\section{Behaviour of Palais-Smale Sequences}
This section examines the behavior of Palais-Smale sequences and shows that these sequences fulfill the Palais-Smale condition below a certain threshold level.\\
Initially, we prove that every Palais-Smale sequence is bounded. This proposition affirms this fact.
\begin{lemma}\label{ps_bounded}
    If $\{ u_k \}$ is a Palais-Smale sequence for the functional (\ref{new_energy}). Then $\{ u_k \}$ is bounded in $\sobo (\om)$.
\end{lemma}
\begin{proof}
Let $\{ u_k \}$ be a Palais-Smale sequence associated with the functional $J_{\lambda}$. Then $J_{\lambda}(u_k) \to c$ for some $c \in \real $ and $J^{\prime}_{\lambda}(u_k) \to 0 $. Thus, we have 
\begin{equation}\label{ps_functional_derivative_0}
\frac 1N\|u_k\|^N - \frac{\lambda}{2 \alpha N}\int_{\om} \frac{|g(u_k)|^{2 \alpha N} ~ dx}{\left( |x| \log \left(\frac{R}{|x|} \right) \right)^N} - \frac12 \int_{\om} \left(\int_{\om} \frac{H(y,g(u_k))}{|x-y|^{\mu}}dy \right)H(x,g (u_k))~dx \to c
\end{equation}
as $k \to \infty$, and 
\begin{align}\label{ps_functional_derivative}
&\left| \int_{\om}{|\nabla u_k(x)|^{N-2}\nabla u_k \nabla v }{dx} - \lambda \int_{\om} \frac{|g(u_k)|^{2 ( \alpha N - 1)} g(u_k) g^{\prime}(u_k) v ~ dx}{\left( |x| \log \left(\frac{R}{|x|} \right) \right)^N} \right. \notag \\
&\left. -\int_\om \left(\int_\om \frac{H(y,g(u_k))}{|x-y|^{\mu}}dy \right)h(x,g(u_k)) g^{\prime}(u_k)v ~dx \right|\leq \e_k\|v\|.
\end{align}
Next, we take $v_k = \frac{g(u_k)}{g^{\prime}(u_k)}$ as a test function in (\ref{ps_functional_derivative}), but we need to show that this belongs to the $\sobo(\om)$. Using $(g_4)$, we get $|v_k| \leq 2 \al |u_k|$ and from direct calculations we have $$ | \nabla v_k | = \left( 1 + (2 \al - 1) \frac{(2\al)^{N-1}(g(u_k))^{N(2\al - 1)}}{1 + (2\al)^{N-1}(g(u_k))^{N(2\al - 1)}}\right) |\nabla u_k| \leq 2 \al | \nabla u_k|.$$ This implies that $v_k \in \sobo(\om)$. Subtituting $v_k$ in (\ref{ps_functional_derivative}), we obtain

\begin{align*}
\left| \langle J_{\lambda}^{\prime}(u_k), v_k\rangle \right| = & \left| \int_{\om}{\left( 1 + (2 \al - 1) \frac{(2\al)^{N-1}(g(u_k))^{N(2\al - 1)}}{1 + (2\al)^{N-1}(g(u_k))^{N(2\al - 1)}}\right)|\nabla u_k(x)|^{N} }{dx} \right.\\
&- \lambda\int_{\om} \frac{|g(u_k)|^{2 \alpha N} }{\left( |x| \log \left(\frac{R}{|x|} \right) \right)^N} \left. -\int_\om \left(\int_\om \frac{H(y,g(u_k))}{|x-y|^{\mu}}dy \right) h(x,g(u_k)) g^{\prime} (u_k) ~dx \right| \\
\leq & 2 \e_k \al \lv u_k \rv .
\end{align*}	

\noindent Now, for some positive constants $C_1$ and $C_2$, we estimate
\begin{align*}
C_1 + C_2 \lv u_k \rv \geq&  ~
2 \theta_{0} J_{\lambda}(u_k) - \langle J_{\lambda}^{\prime} (u_k), v_k) \\
=& \int_{\om}{\left( \frac{2 \theta_{0}}{N} - 1 - (2 \al - 1) \frac{(2\al)^{N-1}(g(u_k))^{N(2\al - 1)}}{1 + (2\al)^{N-1}(g(u_k))^{N(2\al - 1)}}\right)|\nabla u_k(x)|^{N} }{dx} \\
& - \lambda \left( \frac{2 \theta_{0}}{2 \alpha N} - 1 \right)\int_{\om} \frac{|g(u_k)|^{2 \al N} }{\left( |x| \log \left(\frac{R}{|x|} \right) \right)^N}~ dx \\
& -\int_\om \left(\int_\om \frac{H(y,g(u_k))}{|x-y|^{\mu}}dy \right) (\theta_{0} H(x,g(u_k)) - h(x,g(u_k)) g(u_k)) ~dx \\
\geq& \left(  \frac{2 \theta_{0}}{N} - 1 - (2 \alpha - 1) \right) \int_{\om} |\grad u_k|^N - \frac{2 \al \lambda}{C_N} \left( \frac{ \theta_{0}}{ \al N} - 1 \right) \int_{\om} | \grad u_k|^N \\
\geq & ~ 2 \left( \frac{\theta_{0}}{N} -  \al  \right) \left( 1 - \frac{\lambda}{C_N} \right)\lv u_k\rv^N.
\end{align*}
In view of $(h_3)$, we have $ \theta_{0} > \al N $ and $\lambda < C_N$. This gives the boundedness of the sequence $\{ u_k \}$.
\end{proof}

\noindent 
Next, we see that the limit of Palais-Smale sequences turns out to be the weak solution of our problem. To prove this, we prove a series of lemma for each part of the equation.


\begin{lemma}\label{hardy_weak_limit}
Let $\{ u_k \}$ be a Palais-Smale sequence for $J_{\lambda}$ such that $u_k(x) \to u(x)$ pointwise a.e. in $ \om$. Then for any $\phi \in \sobo(\om)$, we have $$\int_{\om} \frac{|g(u_k)|^{2 \al N-2} g(u_k) g^{\prime}(u_k) \phi}{\left( |x| \log \left(\frac{R}{|x|} \right) \right)^N}~dx \to \int_{\om} \frac{|g(u)|^{2 \al N-2} g(u) g^{\prime}(u) \phi}{\left( |x| \log \left(\frac{R}{|x|} \right) \right)^N}~dx.$$
\begin{proof}
    Since $\{ u_k \}$ is a Palais-Smale sequence and $g, g^{\prime}$ are continuous functions, we have 
    $
        u_k(x) \to u(x),  \;\;
        g(u_k(x))  \to g(u(x)),  \;\;
        g^{\prime}(u_k(x))  \to g^{\prime}(u(x)) 
    $
    pointwise a.e. in $\om$ as $k \to \infty$. Then we obtain that $$ \frac{|g(u_k(x)|^{2 \al N-2} g(u_k(x)) g^{\prime}(u_k(x)) }{\left( |x| \log \left(\frac{R}{|x|} \right) \right)^{N-1}} \to \frac{|g(u(x))|^{2 \al N-2} g(u(x)) g^{\prime}(u(x)) }{\left( |x| \log \left(\frac{R}{|x|} \right) \right)^{N-1}}$$ pointwise a.e. in $ \om$ as $k \to \infty$. From $(g_6), (g_{10})$, we deduce that $\left\{ \frac{|g(u_k)|^{2 \al N-2} g(u_k) g^{\prime}(u_k) }{\left( |x| \log \left(\frac{R}{|x|} \right) \right)^{N-1}} \right\} \in L^{\frac{N}{N-1}}(\om)$ is uniformly bounded. Then upto a subsequence it converges weakly to some $w$. But as weak limit and pointwise limit coincides, we obtain $w = \frac{|g(u)|^{2 \al N-2} g(u) g^{\prime}(u) }{\left( |x| \log \left(\frac{R}{|x|} \right) \right)^{N-1}}$.  
    
Moreover, as $\phi \in \sobo (\om)$, then from the Hardy's inequality, we deduce that $ \frac{\phi}{|x| \log \left( \frac{R}{|x|} \right)} \in L^N(\om)$. This implies our assertion.
\end{proof}
\end{lemma}


\noindent From \eqref{ps_functional_derivative_0} and \eqref{ps_functional_derivative}, we deduce that  
	\begin{equation}\label{wk-sol101}
	\int_\om \left(\int_\om\frac{H(y,g(u_k))}{|x-y|^\mu}dy\right)H(x,g(u_k))~dx \leq C_1,
    \end{equation}
    \begin{equation}\label{wk-sol102}
	\int_\om \left(\int_\om\frac{H(y,g(u_k))}{|x-y|^\mu}dy\right)h(x,g(u_k))g^{\prime}(u_k)u_k~dx  \leq C_2,
	\end{equation}
for some constants $C_1$ and $C_2$.\\
Now we establish some results that will help us to deal with the Choquard term in our equation.
\begin{lemma}\label{PS-ws} Assuming $(h_1)$-$(h_5)$ and let $g$ be defined as in \eqref{g}.
If $\{u_k\}$ is a Palais-Smale sequence of $J_{\lambda}$ then there exists $u \in \sobo(\om)$	such that upto a subsequence $ u_k \weak u$ in $\sobo(\om)$ and 
\begin{align}\label{3.23}
	\lim_{k \to \infty} \int_{\om} \left(\int_{\om}\frac{|H(y,g(u_k))-H(y,g(u))|}{|x-y|^{\mu}} dy \right) |H(x,g(u_k))-H(x,g(u))|  ~dx =0.
\end{align}
\end{lemma}
\begin{proof}
From the absolute continuity of $L^1$ functions, we know that if $\mc H\in L^1(\om)$ then for any $\e>0$ there exists a $\delta_\e>0$ such that
		\[\left| \int_{\om'} \mc H(x)~dx\right| <\e,\]
for any measurable set $\om'\subset \om$ with $|\om'|\leq \delta_\e$.

The inequality \eqref{wk-sol101}  implies that
		$$\left(\displaystyle\int_\om\frac{H(y,g(u_k))}{|x-y|^{\mu}}dy\right)H(\cdot,g(u_k)) \in L^{1}(\om).$$
Also, $u \in \sobo(\om)$ gives that
		$$\left(\displaystyle\int_\om\frac{H(y,g(u))}{|x-y|^{\mu}}dy\right)H(\cdot,g(u)) \in L^{1}(\om).$$
We now fix $\de_*>0$, and since we have $ \displaystyle\om = \cup_{\beta \in \real} \{ g(u) \geq \beta \}$, there exists $s'$ such that for $\beta > s'$, we obtain
\begin{align}\label{bb}
			&\int_{\om \cap\{g(u)\geq \beta\}} \left(\int_{\om} \frac{H(y,g(u))}{|x-y|^{\mu}} dy \right)H(x,g(u))  ~dx  \leq \de_*.
\end{align}
Choose $\beta> \max\left\{1,\frac{2 \al C_2 M_0}{{\de_*}}, s', R_0\right\}$. Then by using $(h_2)$, Lemma \ref{L1}-$(g_4)$ and \eqref{wk-sol102}, we deduce
		\begin{align}\label{b}
		&\int_{\om \cap\{g(u_k)\geq \beta\}} \left(\int_{\om} \frac{H(y,g(u_k))}{|x-y|^{\mu}} dy \right)H(x,g(u_k))  ~dx \notag\\&\leq M_0 \int_{\om \cap\{g(u_k)\geq \beta\}} \left(\int_{\om} \frac{H(y,g(u_k))}{|x-y|^{\mu}}dy\right) h(x,g(u_k))  ~dx\notag\\
		&= M_0 \int_{\om \cap\{g(u_k)\geq \beta\}} \left(\int_{\om} \frac{H(y,g(u_k))}{|x-y|^{\mu}}dy\right) \frac{h(x,g(u_k))g(u_k)}{g(u_k)}  ~dx\notag\\
		&\leq \frac{2\al M_0}{ \beta}\int_{\om \cap\{g(u_k)\geq \beta\}} \left(\int_{\om} \frac{H(y,g(u_k))}{|x-y|^{\mu}} dy\right) {h(x,g(u_k)) g^{\prime}(u_k)u_{k}} ~dx<\de_*.
		\end{align}
  Then, by employing \eqref{bb} and \eqref{b}, we have
		\begin{align*}
		&\left| \int_{\om} \left(\int_{\om}\frac{H(y,g(u_k))}{|x-y|^{\mu}} dy \right) H(x,g(u_k)) ~dx-  \int_{\om} \left(\int_{\om}\frac{H(y,g(u) )}{|x-y|^{\mu}} dy \right) H(x,g(u)) ~dx\right|\\
		&\leq 2\de_*+ \left|\int_{\om \cap \{g(u_k)\leq \beta\}} \left(\int_{\om}\frac{H(y,g(u_k))}{|x-y|^{\mu}} dy \right) H(x,g(u_k)) ~dx \right. \\
    &\left. -  \int_{\om\cap \{g(u)\leq \beta\}} \left(\int_{\om}\frac{H(y,g(u))}{|x-y|^{\mu}} dy \right) H(x,g(u)) ~dx\right|.
		\end{align*}
	Next, we show that as $k \to \infty$
		\begin{align*}
		\int_{\om \cap \{g(u_k)\leq \beta\}} &\left(\int_{\om} \frac{H(y,g(u_k))}{|x-y|^{\mu}} dy \right) H(x,g(u_k))~dx  \to \\
  &\int_{\om \cap \{g(u)\leq \beta\}} \left(\int_{\om}\frac{H(y,g(u))}{|x-y|^{\mu}} dy \right) H(x,g(u)) ~dx.
		\end{align*}
	So, to prove this, first, we need pointwise convergence of the integrands. \\ As $\left(\displaystyle\int_\om\frac{H(y,g(u))}{|x-y|^{\mu}}dy\right)H(\cdot,g(u)) \in L^{1}(\om)$,  Fubini's Theorem gives
		\begin{align*}
		&\lim_{\Lambda \to \infty} \int_{\om \cap\{g(u)\leq \beta\}}\left(\int_{\om\cap\{g(u)\geq \Lambda\}}\frac{H(y,g(u))}{|x-y|^{\mu}}dy\right)H(x,g(u))~dx\\
		&= \lim_{\Lambda \to \infty} \int_{\om \cap\{g(u)\geq \Lambda\}}\left(\int_{\om \cap\{g(u)\leq \beta\}}\frac{H(y,g(u))}{|x-y|^{\mu}}dy\right)H(x,g(u))~dx=0.
		\end{align*}
Again, for the same $\de_*$, there exists $s''$ such that for $\Lambda > s''$, we have 
	\[\int_{\om \cap\{g(u)\leq \beta\}} \left(\int_{\om \cap\{g(u)\geq \Lambda\} }\frac{H(y,g(u) )}{|x-y|^{\mu}} dy \right) H(x,g(u)) ~dx \leq \delta_*.\]
Thus, choosing  $\Lambda> \max\left\{1,\frac{2 \al C_2 M_0}{\de_*}, s'', R_0\right\}$ such that, using \eqref{wk-sol102}, $(h_2)$ and Lemma \ref{L1}-$(g_4)$, we deduce
		\begin{align*}
		&\int_{\om \cap\{g(u_k)\leq \beta\}} \left(\int_{\om \cap\{g(u_k)\geq \Lambda\}}\frac{H(y,g(u_k))}{|x-y|^{\mu}} dy \right)H(x,g(u_k)) ~dx\\
		& \leq M_0\int_{\om \cap \{g(u_k)\leq \beta\}} \left(\int_{\om \cap\{g(u_k)\geq \Lambda\}}\frac{ h(y,g(u_k))}{|x-y|^{\mu}} dy \right)H(x,g(u_k)) ~dx\\
		&\leq \frac{M_0}{\Lambda} \int_{\om \cap\{g(u_k)\leq \beta\}} \left(\int_{\om \cap\{g(u_k)\geq \Lambda\} }\frac{  h(y,g(u_k)) g(u_k)(y)}{|x-y|^{\mu}} dy \right) H(x,g(u_k)) ~dx\\
		&\leq \frac{2 \al M_0}{\Lambda} \int_{\om \cap\{g(u_k)\leq \beta\}} \left(\int_{\om \cap\{g(u_k)\geq \Lambda\} }\frac{  h(y,g(u_k))g^{\prime}(u_k) u_k(y)}{|x-y|^{\mu}} dy \right) H(x,g(u_k)) ~dx\\
		&\leq \frac{2 \al M_0}{\Lambda} \int_{\om} \left(\int_{\om }\frac{H(y,g(u_k))}{|x-y|^{\mu}} dy \right) h(x,g(u_k)) g^{\prime}(u_k)u_k ~dx\leq \de_*.
		\end{align*}
Thus, we obtain
		\begin{align*}
		&\left|\int_{\om \cap\{g(u)\leq \beta\}} \left(\int_{\om \cap\{g(u)\geq \Lambda\} }\frac{H(y,g(u) )}{|x-y|^{\mu}} dy \right) H(x,g(u)) ~dx\right.\\
		&\quad \quad \left.- \int_{\om \cap\{g(u_k)\leq \beta\}} \left(\int_{\om \cap\{g(u_k)\geq \Lambda\} }\frac{H(y,g(u_k))}{|x-y|^{\mu}} dy \right) H(x,g(u_k)) ~dx\right|\leq 2\de_*
		\end{align*}
Now we claim that as $k\ra \infty$, for fixed positive real numbers $\beta$ and $\Lambda$ the following holds:
		\begin{equation}\label{choq-new}
		\begin{split}
		&\left|\int_{\om\cap\{g(u_k)\leq \beta\}} \left(\int_{\om \cap\{g(u_k)\leq \Lambda\}}\frac{H(y,g(u_k))}{|x-y|^{\mu}} dy \right) H(x,g(u_k)) ~dx-\right.\\
		&\quad \left.\int_{\om \cap\{g(u)\leq \beta\}} \left(\int_{\om \cap\{g(u)\leq \Lambda\}}\frac{H(y,g(u) )}{|x-y|^{\mu}} dy \right) H(x,g(u)) ~dx \right|\ra 0.
		\end{split}
		\end{equation}
Being $g(u_k)$ and $g(u)$ uniformly bounded, we deduce that
		\begin{align*}
		&\left(\int_{\om \cap\{g(u_k)\leq \Lambda\} }\frac{H(y,g(u_k))}{|x-y|^{\mu}} dy \right) H(x,g(u_k))\chi_{ \om \cap \{g(u_k)\leq \beta\}}\\  &\qquad\ra \left(\int_{\om \cap \{g(u)\leq \Lambda\} }\frac{H(y,g(u))}{|x-y|^{\mu}} dy \right) h(x,g(u))\chi_{\om \cap \{g(u)\leq \beta\}}
		\end{align*}
pointwise a.e. as $k \to \infty$. Now using Lemma \ref{L1}-$(g_5)$ and choosing $r = \al N$ in \eqref{exponential_term_bound}, we get a constant $C_{\beta,\Lambda}>0$ depending on $\beta$ and $\Lambda$ such that
		\begin{align*}
		&\int_{\om \cap \{g(u_k)\leq \beta\}}\left( \int_{\om \cap \{g(u_k)\leq \Lambda\} }\frac{H(y,g(u_k))}{|x-y|^{\mu}} dy \right)  H(x,g(u_k))dx  \\
		&\leq  C_{\beta,\Lambda}\int_{\om \cap \{g(u_k)\leq \beta\}}\left( \int_{\{g(u_k)\leq \Lambda\} }\frac{|g(u_k(y))|^{\al N}}{|x-y|^{\mu}} dy \right)  |g(u_k(x))|^{\al N} dx \\
		&\leq  2 \al C_{\beta,\Lambda}\int_{\om \cap \{g(u_k)\leq \beta\}}\left( \int_{\{g(u_k)\leq \Lambda\} }\frac{|u_k(y)|^{N/2}}{|x-y|^{\mu}} dy \right)  |u_k(x)|^{N/2} dx \\
		& \leq 2 \al C_{\beta,\Lambda} \int_\om\int_{\om }\left(\frac{|u_k(y)|^{N/2}}{|x-y|^{\mu}}~dy  \right) |u_k(x)|^{N/2} ~dx\\
		& \leq 2 \al C_{\beta,\Lambda}C(N,\mu)\|u_k\|_{L^{\frac{N^2}{2N-\mu}}(\om)}^{N}.
		\end{align*}
In the final inequality, we applied the Hardy-Littlewood-Sobolev inequality. Furthermore, using the fact that $u_k \to u$ strongly in $L^q(\om)$ for each $q \in [1,\infty)$ and along with the generalised Lebesgue dominated convergence Theorem \cite{royden2010real}, we obtain \eqref{choq-new}. Consequently, we deduce that 
\begin{equation*}
		\int_{\om} \left(\int_{\om  }\frac{H(y,g(u_k))}{|x-y|^{\mu}} dy \right) H(x,g(u_k)) \ra \int_{\om}\left(\int_{\om }\frac{H(y,g(u))}{|x-y|^{\mu}} dy \right) H(x,g(u)).
\end{equation*}
Therefore, using the same approach above in the proof and the last result, we obtain 
		\begin{align}\label{3.25}
		&\int_{\om} \int_{\{|g(u)| \geq \Lambda\}} \frac{H(y,g(u))}{|x-y|^{\mu}} H(x,g(u)) dy ~dx=o(1),
        \end{align}
        \begin{align}
        &\int_{\om} \int_{\{|g(u_k)| \geq \Lambda\}} \frac{H(y,g(u_k))}{|x-y|^{\mu}} H(x,g(u_k)) dy ~dx = o(1),
		\end{align}
		\begin{align}
		\int_{\om} \int_{\{|g(u)| \geq \Lambda\}} \frac{H(y,g(u_k))}{|x-y|^{\mu}} H(x,g(u)) dy ~dx = o(1),
		\end{align}
		and
		\begin{equation}\label{3.27}
		\int_{\om} \int_{\{|g(u_k)| \geq \Lambda\}} \frac{H(y,g(u_k))}{|x-y|^{\mu}} H(x,g(u)) dy ~dx = o(1) \ \text{as}\ \Lambda \to \infty.
		\end{equation}
		So,
		\begin{equation*}
		\begin{split}
		 \int_{\om} \left(\int_{\om} \right. & \left. \frac{|H(y,g(u_k))-H(y,g(u))|}{|x-y|^{\mu}} dy \right) |H(x,g(u_k))-H(x,g(u))|  ~dx\\\leq&
		2\int_{\om} \left(\int_{\om}\frac{\chi_{\{g(u_k)\geq \Lambda\}}(y)H(y,g(u_k))}{|x-y|^{\mu}} dy \right) H(x,g(u_k)) ~dx \\
		&+4 \int_{\om} \left(\int_{\om}\frac{H(y,g(u_k))\chi_{g(u)\geq \Lambda}(x)H(x,g(u))}{|x-y|^{\mu}} dy \right) ~dx\\&+4 \int_{\om} \left(\int_{\om}\frac{\{\chi_{g(u_k)\geq \Lambda\}}(y)H(y,g(u_k))H(x,g(u))}{|x-y|^{\mu}} dy \right) ~dx\\
		&+2\int_{\om} \left(\int_{\om}\frac{\chi_{g(u)\geq \Lambda}(y)H(y,g(u))}{|x-y|^{\mu}} dy \right) H(x,g(u)) ~dx \\
		&+\int_\om\Bigg[\left(\int_\om\frac{|H(y,g(u_k))\chi_{g(u_k)\leq \beta}-H(y,g(u))\chi_{\{g(u)\leq \Lambda\}}|}{|x-y|^{\mu}}dy\right)\\&\qquad\qquad\qquad|H(x,g(u_k))\chi_{g(u_k)\leq \beta}-H(x,g(u))\chi_{\{g(u)\leq \Lambda\}}|\Bigg]~dx.
		\end{split}
		\end{equation*}
Then, from Lebesgue's dominated convergence Theorem, we infer that the last integration tends to $0$ as $k \to \infty.$ Hence, using \eqref{3.25}-\eqref{3.27}, we finally conclude \eqref{3.23}.	
\end{proof}

\noindent We use the following proposition borrowed from the \cite{boccaaooj1992almost} to get the pointwise convergence of the gradients.
\begin{proposition}\label{pointwise_gradient}
Let $\om \subset \rnn$ be a bounded domain. If $\{ u_k \}$ be a bounded sequence in $\sobo(\om)$ such that $u_k \weak u$ in $ \sobo(\om)$ and satisfies $$ - \Delta_N u_k = f_k + g_k \mbox{ in } D^{\prime} (\om)$$ where $f_k \to f$ in $W^{-1,p^{\prime}}(\om)$ and $g_k$ is a bounded sequence of Radon Measures, i.e., $$\langle g_k, \phi \rangle \leq C_K \lv \phi \rv_{\infty}$$ for all $\phi \in C_c^{\infty}(K)$ for each of the compact sets $K \subseteq \om$. Then there exists a subsequence $\{u_{k_n}\}$ of $\{u_k\}$ such that $$ \grad u_{k_n}(x) \to \grad u(x) \mbox{ a.e. as } n \to \infty \mbox{ in } \om.$$
\end{proposition}
\begin{lemma}\label{kc_ws} Assume conditions $(h_1)$ through $(h_6)$ are satisfied, and let $g$ be defined as in equation \eqref{g}. If $\{ u_k \}$ is a Palais-Smale sequence for $J_{\lambda}$. Then we have pointwise convergence of the gradients, i.e., $\grad u_k(x) \to \grad u(x) $ a.e. in $\om$ as $k\to\infty$. This leads to the conclusion that 
\begin{equation*}
		|\nabla u_k|^{N-2}\nabla u_k \rightharpoonup |\nabla u|^{N-2}\nabla u\; \text{weakly in}\; (L^{\frac{N}{N-1}}(\om))^N
\end{equation*}
as $k \to \infty$.
Furthermore, for all $w\in \sobo(\om),$  as $k\to\infty$, we have
		\begin{align*}\int_\om \left(\int_\om \frac{H(y, g(u_k))}{|x-y|^{\mu}}dy\right)h(x, g(u_k))g^{\prime}(u_k)w dx\to \int_\om \left(\int_\om \frac{H(y,g(u))}{|x-y|^{\mu}}dy\right)h(x,g(u))g^{\prime}(u)wdx.  \end{align*}
\end{lemma}
\begin{proof}
		
Let $\om' \subset\subset \om$ and $\psi \in C_c^\infty(\om)$ such that $0\leq \psi \leq 1$ and $\psi \equiv 1$ in $\om' $. One can easily compute that
		\begin{equation}\label{k0}
		\begin{split}
		\left\| \frac{\psi}{1+u_k}\right\|^N = \int_\om \left|\frac{\nabla \psi}{1+u_k}- \psi \frac{\nabla u_k}{(1+u_k)^2} \right|^N~dx
		\leq 2^{N-1}(\|\psi\|^N+ \|u_k\|^N),
		\end{split}
		\end{equation}
which yields that $\frac{\psi}{1+u_k} \in W^{1,N}_0(\om)$. Now taking $v=\frac{\psi}{1+u_k}$  in \eqref{ps_functional_derivative} as a test function and using Lemma \ref{L1}-$(g_6), (g_{10})$ and \eqref{k0}, we obtain
		\begin{equation}\label{kc-ws-new1}
		\begin{split}
		&\int_{\om^{'}}\left( \int_\om \frac{H(y,g(u_k))}{|x-y|^{\mu}}dy\right)\frac{h(x,g(u_k))}{1+u_k}g^{\prime}(u_k)~dx\\
        & \leq \int_\om \left( \int_\om \frac{H(y,g(u_k))}{|x-y|^{\mu}}dy\right)\frac{h(x,g(u_k)) g^{\prime}(u_k)\psi}{1+u_k}~dx\\
		&\leq  \e_k \left\|\frac{\psi}{1+u_k}\right\| + \lambda \int_{\om} \frac{|g(u_k)|^{2 \al N-2} g(u_k) g^{\prime}(u_k) \psi}{\left( |x| \log \left(\frac{R}{|x|} \right) \right)^N(1+u_k)}  + \int_\om  |\nabla u_k|^{N-2}\nabla u_k \nabla \left( \frac{\psi}{1+u_k}\right)~dx\\
		& \leq \e_k 2^{\frac{N-1}{N}}(\|\psi\|+ \|u_k\|)+ \lambda \int_{\om} \frac{|u_k|^{N-1}\psi}{\left( |x| \log \left(\frac{R}{|x|} \right) \right)^N} \\
  & \hspace{0.5cm }+\int_\om |\nabla u_k|^{N-2}\nabla u_k \left(\frac{\nabla \psi}{1+u_k}-\psi\frac{\nabla u_k}{(1+u_k)^2}\right)~dx\\ 
		& \leq \e_k 2^{\frac{N-1}{N}}(\|\psi\|+ \|u_k\|) + \frac{ \lambda}{C_N} \lv u_k\rv^{N-1} \lv \psi \rv+\int_\om |\nabla u_k|^{N-1} \left( |\nabla \psi|+ |\nabla u_k|\right)~dx\\
		& \leq \e_k 2^{\frac{N-1}{N}}(\|\psi\|+ \|u_k\|)+ \left( 1 + \frac{ \lambda}{C_N}\right)\|\psi\|\|u_k\|^{N-1}+ \|u_k\|^N\leq C',
		\end{split}
		\end{equation}
where $C'$ is a positive constant and in the last line we used the fact that $\{u_k\}$ is bounded in $W^{1,N}_0(\om)$.
Again, using the boundedness of the sequence $\{u_k\}$, from \eqref{ps_functional_derivative}, we get
		\begin{align}\label{kc-ws-new2}
		&\int_{\om^{'}} \left( \int_\om \frac{H(y,g(u_k))}{|x-y|^{\mu}}dy\right){h(x,g(u_k))g^{\prime}(u_k)}{u_k}~dx\notag\\
		&\leq \int_{\om} \left( \int_\om \frac{H(y,g(u_k))}{|x-y|^{\mu}}dy\right){h(x,g(u_k))g^{\prime}(u_k)}{u_k}~dx \notag\\
		&\leq \e_k\|u_k\| + \left( 1 - \frac{ \lambda}{C_N}\right)\|u_k\|^N\leq C_2
		\end{align}
for some constant $C_2>0$. Combining \eqref{kc-ws-new1} and \eqref{kc-ws-new2}, we deduce
		\begin{align*}
		\int_{\om^{'}}\left( \int_\om \right. & \left. \frac{H(y,g(u_k))}{|x-y|^{\mu}}dy \right){h(x,g(u_k))g^{\prime}(u_k)}~dx \\
		 \leq & 2\int_{\om^{'}\cap \{u_k <1\}} \left( \int_\om\frac{H(y,g(u_k))}{|x-y|^{\mu}}dy\right)\frac{h(x,g(u_k))g^{\prime}(u_k)}{1+u_k}~dx\\   
        & + \int_{\om^{'}\cap \{u_k \geq 1\}} \left(\int_\om \frac{H(y,g(u_k))}{|x-y|^{\mu}}dy\right){h(x,g(u_k))g^{\prime}(u_k) u_k}~dx\\
		 \leq & 2\int_{\om^{'}} \left( \int_\om\frac{H(y,g(u_k))}{|x-y|^{\mu}}dy\right)\frac{h(x,g(u_k))g^{\prime}(u_k)}{1+u_k}~dx\\
        & + \int_{\om^{'}}\left( \int_\om \frac{H(y,g(u_k))}{|x-y|^{\mu}}dy\right){h(x,g(u_k))g^{\prime}(u_k) u_k}~dx\\
		 \leq & 2C_1+C_2  :=C_3.
		\end{align*}
Thus, the sequence $\{w_k\}:=\left\{\left( \int_\om\frac{H(y,g(u_k))}{|x-y|^{\mu}}dy\right){h(x,g(u_k))g^{\prime}(u_k)}\right\}$ is bounded in $L^1_{\text{loc}}(\om)$. \\
Now, $u_k$ satisfies the following equation weakly
$$ - \Delta_N u_k = J_{\lambda}^{\prime}(u_k) + \lambda \frac{|g(u_k)|^{2 \al N-2} g(u_k) g'(u_k)}{\left( |x| \log\left( \frac{R}{|x|} \right) \right)^N} + \left(\int_\om \frac{H(y,g(u_k))}{|x-y|^{\mu}}dy \right)h(x,g(u_k)) g^{\prime}(u_k).$$
Taking $f_k = J_{\lambda}^{\prime}(u_k)$ and $g_k = \lambda \frac{|g(u_k)|^{2 \al  N-2} g(u_k) g'(u_k)}{\left( |x| \log\left( \frac{R}{|x|} \right) \right)^N} + \left(\int_\om \frac{H(y,g(u_k))}{|x-y|^{\mu}}dy \right)h(x,g(u_k)) g^{\prime}(u_k) $.
Using the proposition \ref{pointwise_gradient}, we conclude that $\grad u_k(x) \to \grad u(x) ~a.e.$ in $\om$. \\

\noindent Now, as $\{ u_k \}$ is bounded, there exists a radon measure $\zeta^*$ such that, up to a subsequence, $u_k \rightharpoonup \zeta^*$ in the ${weak}^*$-topology as $k \to \infty$.  Hence,  we have
		\[\lim_{k \to \infty}\int_\om\int_\om \left( \frac{H(y,g(u_k))}{|x-y|^{\mu}}dy\right){h(x,g(u_k))g^{\prime}(u_k)}\eta ~dx = \int_\om \eta ~d\zeta^*\; \forall \eta \in C_c^\infty(\om). \]
Since $u_k$ satisfies \eqref{ps_functional_derivative}, we achieve
		\[\int_A \eta d\zeta^*= \lim_{k \to \infty} \left( \int_A |\nabla u_k|^{N-2}\nabla u_k \nabla \eta ~dx + \int_{A} \frac{|g(u_k)|^{2 \al N-2} g(u_k) g^{\prime}(u_k) \eta}{\left( |x| \log \left(\frac{R}{|x|} \right) \right)^N}~ dx\right) \;\;\forall\; A\subset \om,  \]
		which together with Lemma \ref{hardy_weak_limit} and weak convergence of $\{ |\grad u_k|^{N-2}\grad u_k \}$ yields that the Radon measure $\zeta^*$ is absolutely continuous with respect to the Lebesgue measure. So, there exists a function $\psi \in L^1_{\text{loc}}(\om)$ such that for any $\eta\in C^\infty_c(\om)$, it holds that $\int_\om \eta~ d\zeta^*= \int_\om \eta \psi~dx$, thanks to Radon-Nikodym Theorem.\\
Therefore, we obtain
		\begin{align*}&\lim_{k \to \infty}\int_\om\left( \int_\om \frac{H(y,g(u_k))}{|x-y|^{\mu}}dy\right){h(x,g(u_k))g^{\prime}(u_k)}\eta(x)~dx\\&\qquad = \int_\om \eta \psi~dx= \int_\om  \left( \int_\om \frac{H(y,g(u))}{|x-y|^{\mu}}dy\right){h(x,g(u))g^{\prime}(u)}\eta(x)~dx\;\;\forall\; \eta\in C^\infty_c(\om), \end{align*}
since $ C^\infty_c(\om)$ is dense in $W_0^{1,N}(\om)$, this completes the proof.
\end{proof}


\noindent \textbf{Proof of the Theorem \ref{main_theorem}:}
By the Mountain Pass Theorem, there exists a Palais-Smale sequence $\{u_k\}$ at level $\beta^*(\la)$. According to Lemma \ref{ps_bounded}, it follows that $u_k \weak u$ in $\sobo(\om)$ as $k \to \infty$. From Lemmas \ref{kc_ws} and \ref{hardy_weak_limit}, it follows that $u$ is a weak solution of problem \eqref{pq_new}.\\
Next, we have to show that $u \not \equiv 0$. On the contrary, assume that $ u \equiv 0$. Lemma \ref{PS-ws} give 
\[\int_\om \left(\int_\om \frac{ H(y,g(u_k))}{|x-y|^{\mu}}dy\right)H(x,g(u_k)) ~dx  \to 0\; \text{as}\; k \to \infty,\]
Applying this together with $\ds\lim_{k\ra \infty} J_{\la} (u_k) = \beta^*(\la)$, $(g_6), (g_{10})$ and Hardy inequality, we obtain $$ \left( 1 - \frac{  \lambda}{C_N} \right) \lim_{k\ra \infty}\frac{\lv u_k \rv^N}{N} \leq \lim_{k\ra\infty}\left[\frac{\lv u_k \rv^N}{N} -  \frac{\lambda}{2 \al N} \int_{\om} \frac{|g(u_k)|^{2 \al  N}}{\left( |x| \log \left( \frac{R}{|x|} \right) \right)^N}\right] = \beta^*(\la).$$
Thus for $k$ large enough, we obtain 
\begin{align}\label{k}
\frac{\lv u_k \rv^N}{N} \leq \ba^*(\la) \left( 1 - \frac{\lambda}{C_N} \right)^{-1}\leq  \ba^{*}(0) \left( 1 - \frac{\lambda}{C_N} \right)^{-1}<L,
\end{align}
for any $0 \leq \lambda < \frac{C_N(L  - \tau)}{ L}:=\lambda_{0}$, where $\tau = \beta^*(0)$, $L= \frac{1}{2 \al N} \left( \frac{2 N -\mu}{2N} \al_N\right)^{N-1}$. 
Furthermore, we show that 
\begin{align}\label{aaa}
	\int_{\om}\left( \int_\om\frac{H(y,g(u_k))}{|x-y|^{\mu}}dy\right)h(x,g(u_k))g^{\prime}(u_k)u_k~dx \to 0 \;\text{as}\; k \to \infty.
\end{align}
For this, $(g_4), (h_3), (h_5)$ and inequality (\ref{hls_inequality}) yield
\begin{align*}
	&\int_{\om}\left( \int_\om\frac{H(y,g(u_k))}{|x-y|^{\mu}}dy\right)h(x,g(u_k))g^{\prime}(u_k)u_k dx\notag\\
	&\leq\frac{1}{\theta_0}\int_{\om}\left( \int_\om\frac{h(y,g(u_k)) g(u_k)}{|x-y|^{\mu}}dy\right)h(x,g(u_k))g(u_k) dx\notag \\
	&\leq \frac{C_{t,N,\mu,r}}{\theta_0} \left(\int_{\om}|h(x,g(u_k))g(u_k)|^{\frac{2N}{2N-\mu}}~dx\right)^{\frac{2N-\mu}{N}}\notag\\
 &\leq C(\e,N,\al,r) \left(\int_{ \om}\left(|g(u_k)|^{\al N} +  |g(u_k)|^r \exp((1+\e)|g(u_k)|^{\frac{2\al N}{N-1}})\right)^{\frac{2N}{2N-\mu}}~dx\right)^{\frac{2N-\mu}{N}}\notag \\
	&  \leq C(\e,N,\al,r)\left( \|u_k\|_{L^{\frac{ N^2}{2N-\mu}}}^{N}+  \underbrace{\left(\int_{ \om}|g(u_k)|^\frac{2rN}{2N - \mu} \exp\left((1+\e)|g(u_k)|^{\frac{2\al N}{N-1}}\frac{2N}{2N-\mu}\right)~dx\right)^{\frac{2N-\mu}{N}}}_{I}\right)  \notag \\
\end{align*}
applying the H\"{o}lder inequality on $I$ with $q$ and $q^{\prime}$ being the conjugates, we obtain 
\begin{align*}
    I &\leq  \lv u_k \rv_{L^{\frac{2rq^{\prime}N}{2N - \mu}}}^{2r} \left(\int_{ \om} \exp\left(q(1+\e)|g(u_k)|^{\frac{2\al N}{N-1}}\frac{2N}{2N-\mu}\right)~dx\right)^{\frac{2N-\mu}{qN}} \\
    & \leq  \lv u_k \rv_{L^{\frac{2rq^{\prime}N}{2N - \mu}}}^{2r} \left(\int_{ \om} \exp\left(q(1+\e)(2\al)^{\frac{1}{N-1}} \frac{2N}{2N-\mu} \lv u_k\rv^{\frac{N}{N-1}} \left( \frac{|u_k|}{\lv u_k \rv}\right)^{\frac{N}{N-1}}\right)~dx\right)^{\frac{2N-\mu}{qN}}.
\end{align*}
In view of  Trudinger-Moser inequality (\ref{TM-ineq}), the exponential integral is finite if and only if $$q(1+\e)(2\al)^{\frac{1}{N-1}} \frac{2N}{2N-\mu} \lv u_k\rv^{\frac{N}{N-1}} \leq \al_N $$
that is $$\frac{\lv u_k \rv^N}{N} \leq \frac{1}{2 \al N} \left( \frac{2N - \mu}{2N} \alpha_N \right)^{N-1}  \left( \frac{1}{q(1+\e)} \right)^{N-1},$$
by choosing $q$ close to $1$ and $\e$ close to $0$ then (\ref{k}) implies that the integral is finite. Thus, our claim (\ref{aaa}) follows.\\
As $\langle J_{\lambda}^{\prime} (u_k),u_k\rangle \to 0$, we deduce that for all $|\lambda| < \lambda_0$, 
\begin{align*}
    \int_{\om} |\grad u_k|^N -\lambda \int_{\om} \frac{|g(u_k)|^{2 \al N-2}g(u_k) g^{\prime} (u_k) u_k}{\left( |x| \log \left( \frac{R}{|x|} \right) \right)^N} = o_k(1) 
\end{align*}
which implies 
\begin{align*}
    0 \leq \left( 1 - \frac{ \lambda }{C_N}\right) \lv u_k \rv^N \leq  \int_{\om} |\grad u_k|^N -\lambda \int_{\om} \frac{|g(u_k)|^{2 \al N-2}g(u_k) g^{\prime} (u_k) u_k}{\left( |x| \log \left( \frac{R}{|x|} \right) \right)^N} = o_k(1). 
\end{align*}
Therefore, we have $\ds\lim_{k \to \infty} \lv u_k \rv^N = 0$. But this implies that $\beta^*(\lambda) \leq \lim_{k \to \infty}\left( 1 + \frac{\lambda}{C_N} \right) \lv u_k \rv^N$, i.e., $\beta^*(\lambda) = 0$. This is a contradiction. Thus, $w \not \equiv 0$. \\
Now, as $u$ is a weak solution to \eqref{pq_new}. Thus, for all $\phi \in \sobo(\om)$, we have $ \langle J_{\la}^{\prime}(u),\phi \rangle = 0$. Putting $\phi = u^-$, and using $(g_6), (g_{10}),$ and Hardy inequality, we deduce
$$\left( 1 - \frac{\lambda}{C_N}\right) \int_{\om} |u^-|^N dx \leq 0. $$ Thus, $u \geq 0$ a.e. in $\om$. By the standard regularity theory in \cite{dibenedetto_regularity, tolksdorf_regularity}, we conclude $u \in C^{1,\alpha_0}(\om \setminus \{ 0 \})$ for some $\alpha_0 \in (0,1)$. Then by the strong maximum principle in \cite{vazquez_strong_maximum}, we deduce that $u > 0$ in $\om \setminus \{ 0 \}$.

\section{Non-Homogeneous Problem}
In this section, we prove the Theorem \ref{main_theorem_non_homo}. For this, consider the energy functional corresponding to non-homogeneous problem  (\ref{pq_non_homo}):
\begin{equation*}
\begin{split}
\tilde{J}_{\lambda, NH}(w)=& \frac{1}{N}\displaystyle\int_{\om }(1+(2\alpha)^{N-1}|w|^{N(2\al -1)})|\nabla w|^{N}dx  -\frac{\lambda}{2 \al N} \int_{\om} \frac{|w|^{2 \al N}}{\left( |x| \log\left( \frac{R}{|x|} \right) \right)^N}dx \\ & -\frac 12\int_{\om}\int_{\om} \frac{H(y,w(y))H(x,w(x))}{|x-y|^\mu}dxdy + \ka \int_{\om} q(x) w(x) ~ dx.
\end{split}
\end{equation*}
Using the same transformation $g(t)$ in \eqref{g}, the new functional is given by:
\begin{align*}
J_{\lambda, NH}(u)=& \frac{1}{N}\displaystyle\int_{\om }|\nabla u|^{N}dx -\frac{\lambda}{2 \al N} \int_{\om} \frac{|g(u)|^{2 \al N}}{\left( |x| \log\left( \frac{R}{|x|} \right) \right)
^N}~dx \notag \\ & -\frac 12\int_{\om}\int_{\om} \frac{H(y,g(u(y)))H(x,g(u(x)))}{|x-y|^\mu}dxdy + \kappa \int_{\om} q(x) g(u) dx.
\end{align*}
Now, this new energy functional is $C^1(W_0^{1,N},\real)$ and its derivative is given as:
\begin{align*}
\langle J^{\prime}_{\lambda, NH}(u),v \rangle=& \displaystyle\int_{\om }|\nabla u|^{N-1} \na u  \na v \, dx - \lambda \int_{\om} \frac{|g(u)|^{2\al N-2} g(u) g^{\prime}(u) v}{\left( |x| \log\left( \frac{R}{|x|} \right) \right)^N}~dx  \\ &- \int_{\om}\int_{\om} \frac{H(y,g(u(y)))h(x,g(u(x))) g^{\prime}(u(x))}{|x-y|^\mu}dxdy + \ka \int_{\om} q(x) g^{\prime}(u) v.
\end{align*}
The critical points of the new functional $J_{\lambda, NH}$ are the weak solutions of the following problem:
\begin{equation*}\label{pq_new_non_homo}
\tag{$ P_{\lambda,NH}^*$}\left\{
\begin{array}{l}
 \mathcal{L}_1(u) 
 = \left(\int_{\om} \frac{H(y,g(u(y)))}{|x-y|^{\mu}}dy\right) h(x,g(u((x)))g^{\prime}(u(x)) + \kappa q(x) g^{\prime}(u) \mbox{ in } \om , \\
u>0 \;\mbox{in} \; \om\setminus\{0\} \quad  u = 0 \mbox{ on } \partial \om,
\end{array}
\right.
\end{equation*}
where $\mathcal{L}_1(u) = - \nlap u  - \lambda \frac{|g(u)|^{2 \al N-2}g(u)g^{\prime}(u)}{\left( |x| \log\left(\frac{R}{|x|} \right)^N \right)} $.
The problems (\ref{pq_non_homo}) and (\ref{pq_new_non_homo}) are equivalent as one can easily see. Also, $u$ is a weak solution of (\ref{pq_new_non_homo}) if and only if $w = g(u)$ is a weak solution of (\ref{pq_non_homo}).\\
Now, to prove the Theorem \ref{main_theorem_non_homo}, we need some basic lemmas which we prove below:
\begin{lemma}
    For any $u \in \sobo (\om)$, we have $J_{\lambda, NH}(tu) \to -\infty$ as $t \to \infty$.
\end{lemma}
\begin{proof}
Using $(g_5)$, H\"{o}lder inequality and Sobolev inequality, we deduce 
\begin{align}\label{k1}
 \int_{\om} q(x) g(u(x)) \, dx \leq C_3  \lv q \rv_{L^{\frac{N}{N-1}}(\om)} \lv u\rv.
\end{align}
Using $(g_5)$, \eqref{k1}, and following the similar lines of Lemma \ref{geometry_of_functional_at_infinity}, we get $k \in \ntrl$ and constants $C_1, C_2 > 0$ depending on $k$, such that
\begin{align*}
    J_{\lambda, NH}(tu) \leq & \frac{t^N}{N} \int_{\om} | \nabla u|^N  - \frac{1}{2} \int_{\om} \left(\int_{\om} \frac{H(y, g(tu))}{|x-y|^{\mu}}dy\right)H(x, g(t u))~dx + \ka \int_{\om} q(x) g(tu(x)) \, dx \\
    \leq & \frac{t^N}{N} \int_{\om} | \nabla u|^N - C_1^2 t^{\frac{k}{\al}}  \int_{\om} \int_{\om} \frac{|u(y)|^{\frac{k}{2 \al}} |u(x)|^{\frac{k}{2 \al}}}{|x-y|^{\mu}}~dxdy   +2C_1C_2 t^{\frac{k}{2 \al}}\int_{\om} \int_{\om}\frac{|u(x)|^{\frac{k}{2 \al}}}{|x-y|^{\mu}}~dxdy \\
    & - C_2^2 \int_{\om} \int_{\om} \frac{dx~dy}{|x-y|^{\mu}}+  C_3 t \ka \lv q \rv_{L^{\frac{N}{N-1}}(\om)} \lv u\rv.
\end{align*}
Choosing $k> \al N$, completes the proof of the Lemma.
\end{proof}

\noindent Next, we see that the functional possesses the remaining mountain pass geometry near the origin.
\begin{lemma}
Assuming $(h_1)-(h_5)$. Then there exists $\delta, \rho > 0$ such that $$ J_{\lambda, NH}(u) \geq \delta \mbox{ whenever } \lv u \rv  = \rho.$$
\end{lemma}
\begin{proof}
Lemma \ref{geometry_of_functional_near_zero} and \eqref{k1} yield 
\begin{align*}
J_{\lambda, NH}(u) \geq & \frac{1}{N}\left( 1 - \frac{ \lambda}{C_N}\right) \lv u \rv ^N -  C_1 \lv u \rv ^{2 \al N} - C_2 \lv u \rv ^{2r} - \ka \lv q \rv_{L^{\frac{N}{N-1}}(\om)} \lv u\rv_{L^N(\om)} \\
&  = \lv u \rv \left[ \frac{1}{N}\left( 1 - \frac{\lambda}{C_N}\right) \lv u \rv ^{N-1} -  C_1 \lv u \rv ^{2 \al N-1} - C_2 \lv u \rv ^{2r-1} - \ka C_3 \lv q \rv_{L^{\frac{N}{N-1}}(\om)}  \right]
\end{align*}
where $C_1$ and $C_2$ are some constants depending on $N, r, \al,$ and $\e$. So, we choose $2r > N, \ka > 0$ small enough. Thus, there exists $\rho>0$ such that if $\lv u \rv  = \rho$, then $J_{\lambda, NH}(u) \geq \delta$.
\end{proof}

\noindent Thus, the functional satisfies the mountain pass geometry.

\begin{lemma}\label{ps_bounded_non_homo}
    If $\{ u_k \}$ is a Palais-Smale sequence for the functional $J_{\lambda,NH}$. Then $\{ u_k \}$ is bounded in $\sobo (\om)$.
\end{lemma}
\begin{proof}
Let $\{ u_k \}$ be a Palais-Smale sequence associated with the functional $J_{\lambda,NH}$. Then $J_{\lambda, NH}(u_k) \to c$ for some $c \in \real $ and $J_{\lambda, NH}^{\prime}(u_k) \to 0 $. Thus, we have
\begin{align*}
C_1 + C_2 \lv u_k \rv \geq & ~2 \theta_{0} J_{\lambda, NH}(u_k) - \left\langle J_{\lambda, NH}^{\prime} (u_k), \frac{g(u_k)}{g^{\prime}(u_k)}\right\rangle\\
 \geq & ~ 2 \left( \frac{\theta_{0}}{N} -  \al  \right) \left( 1 - \frac{\lambda}{C_N} \right)\int_{\om} |\grad u_k|^N + \ka (2 \theta_{0} - 1) \int_{\om} q(x) g(u_k) dx\\
 \geq & ~\lv u_k \rv \left[ 2 \left( \frac{\theta_{0}}{N} -  \al  \right) \left( 1 - \frac{\lambda}{C_N} \right)\lv u_k \rv^{N-1} - \ka (2 \theta_{0} -1) C_3 \lv q \rv_{L^{\frac{N}{N-1}}(\om)}\right],
\end{align*}
where $C_3$ is some constant. In view of $(h_3)$, we have $ \theta_{0} >  \al N$. This gives the boundedness of the sequence $\{ u_k \}$ for small values of $\ka$.
\end{proof}

\noindent \textbf{Proof of the Theorem \ref{main_theorem_non_homo}:} 
    Since the functional $J_{\lambda, NH}$ corresponding to the equation \eqref{pq_new_non_homo} satisfies the mountain pass geometry. Using the Mountain Pass Theorem in \cite{rabinowitz1986minimax}, we obtain a Palais-Smale sequence $\{u_k\}$. In view of Lemma \ref{ps_bounded_non_homo}, $\{ u_k \}$ is bounded and possesses one weakly convergent subsequence $\{ u_{k_n}\}$ in $\sobo(\om)$ which converges to $u \in \sobo(\om)$. We conclude that $u$ is a non trivial weak solution by Lemmas \ref{hardy_weak_limit} and \ref{kc_ws}.\\

\noindent{\bf Acknowledgment:} 
The second author would like to thank the Science and Engineering Research Board, Department of Science and Technology, Government of India for the financial support under the grant SPG/2022/002068.\\

\noindent \textbf{Conflict of interest:} All authors certify that there is no actual or potential conflict of interest in relation to this article.

\bibliographystyle{abbrv}
\bibliography{quasilinear}
\end{document}